\documentclass[reqno,a4paper]{amsart}
\usepackage{eucal,amsfonts,amssymb,amsmath,amsthm,epsfig,mathrsfs}
\usepackage{xcolor}

\usepackage{amscd,amsxtra}
\usepackage{enumerate}
\usepackage{latexsym}

\allowdisplaybreaks

 \makeatletter \@addtoreset{equation}{section}

\makeatother \makeatletter

\newtheorem{thm}{Theorem}[section]
\newtheorem{hyp}[thm]{Hypothesis}{\rm}
\newtheorem{lemm}[thm]{Lemma}
\newtheorem{coro}[thm]{Corollary}
\newtheorem{prop}[thm]{Proposition}
\newtheorem{defi}[thm]{Definition}
\newtheorem{rmk}[thm]{Remark}

\newcommand{\R}{{\mathbb R}}
\newcommand{\N}{{\mathbb N}}

\newcommand{\Rd}{\mathbb R^d}

\newcommand{\supp}{{\rm{supp}}\,}

\newcommand{\bd}{\begin{defi}}
\newcommand{\ed}{\end{defi}}
\newcommand{\nnm}{\nonumber}
\newcommand{\be}{\begin{equation}}
\newcommand{\ee}{\end{equation}}
\newcommand{\barr}{\begin{array}}
\newcommand{\earr}{\end{array}}
\newcommand{\bmn}{\begin{eqnarray}}
\newcommand{\emn}{\end{eqnarray}}
\newcommand{\bnm}{\begin{eqnarray*}}
\newcommand{\enm}{\end{eqnarray*}}
\newcommand{\bln}{\begin{subequations}}
\newcommand{\eln}{\end{subequations}}

\newcommand{\ba}{\begin{align}}
\newcommand{\ea}{\end{align}}
\newcommand{\banm}{\begin{align*}}
\newcommand{\eanm}{\end{align*}}
\newcommand{\one}{\mbox{$1\!\!\!\;\mathrm{l}$}}

\title[Compactness and invariance properties of evolution operators]{Compactness and invariance properties of evolution operators associated with Kolmogorov operators with unbounded coefficients
}
\author{L. Angiuli}
\address{Dipartimento di Matematica ``Ennio De Giorgi'', Universit\`a del Salento, Via Per Arnesano, C.P. 193, I-73100, Lecce, Italy.}
\email{luciana.angiuli@unisalento.it}
\author{L. Lorenzi}
\address{Dipartimento di Matematica, Universit\`a degli Studi di Parma, Parco Area delle Scienze 53/A, I-43124 Parma, Italy.}
\email{luca.lorenzi@unipr.it}

\begin{document}

\begin{abstract}
In this paper we consider nonautonomous elliptic operators ${\mathcal A}$ with nontrivial potential term defined in $I\times\R^d$, where $I$ is
a right-halfline (possibly $I=\R$). We prove that we can associate an
evolution operator $(G(t,s))$ with ${\mathcal A}$ in the space of all bounded and continuous functions on $\Rd$. We also study the compactness properties of the operator $G(t,s)$. Finally, we provide sufficient conditions guaranteeing that each operator $G(t,s)$ preserves
the usual $L^p$-spaces and $C_0(\Rd)$.

\end{abstract}

\maketitle
\section{Introduction}
Second-order autonomous elliptic operators with unbounded
coefficients have been the subject of many mathematical researches.
The interest in such operators comes from their many  applications to branches of life sciences such as mathematical finance.
Starting from the pioneering papers by It\^o \cite{ito} and Azencott \cite{azencott},
the literature has spread out considerably and an almost systematic
treatment of such operators (and their associated semigroups) is nowadays available. We refer the reader to e.g., \cite{luciana-giorgio-chiara,BerLor07Ana,MetPalWac02} and their rich bibliographies.

On the contrary the study of nonautonomous second-order elliptic
operators is at a preliminary level. The pioneering paper is
\cite{DaPLun07} where the nonautonomous Ornstein-Uhlenbeck operator
\begin{eqnarray*}
({\mathcal
L}(t)\varphi)(x)=\sum_{i,j=1}^dq_{ij}(t)D_{ij}\varphi(x)+\sum_{i,j=1}^db_{ij}(t)x_jD_i\varphi(x),\qquad\;\,(t,x)\in\R^{1+d},
\end{eqnarray*}
has been studied in the case when its coefficients are $T$-periodic
for some $T>0$. The analysis of \cite{DaPLun07} has been continued
in a couple of papers by Geissert and Lunardi (see
\cite{GeisLun08,GeisLun09}) where $\mathcal{L}$ and the associated evolution operator
$(L(t,s))$ have been extensively studied both in periodic and nonperiodic
settings.

Recently, in \cite{KunLorLun09Non} the more general
nonautonomous elliptic operator
\begin{eqnarray*}
({\mathcal
A}(t)\varphi)(x)=\sum_{i,j=1}^dq_{ij}(t,x)D_{ij}\varphi(x)+\sum_{j=1}^db_j(t,x)D_j\varphi(x),\qquad\;\,(t,x)\in I\times\R^d,
\end{eqnarray*}
has been studied, when $I$ is a right-halfline (possibly $I=\R$). Under rather mild regularity conditions on its
coefficients and assuming the ellipticity condition
\begin{eqnarray*}
\sum_{i,j=1}^dq_{ij}(t,x)\xi_i\xi_j\ge\eta_0|\xi|^2,\qquad\;\,(t,x)\in I\times\Rd,\;\,\xi\in\Rd,
\end{eqnarray*}
for some positive constant $\eta_0$, the existence of a (unique) evolution operator $(G(t,s))$ associated with $\mathcal{A}$ in $C_b(\Rd)$ (the
 space of all bounded and continuous functions $f:\Rd\to\R$) has been proved.
The main properties of the evolution operator $G(t,s)$ in $C_b(\Rd)$ have been extensively studied and the authors extended many of
the results proved for the Ornstein-Uhlenbeck operator.

In the autonomous case it is well known that, in general, the semigroups associated with elliptic operators with
unbounded coefficients do not well behave in the usual $L^p$-spaces related to the Lebesgue measure. On the contrary,
they enjoy nice properties in the $L^p$-spaces related to the so-called invariant measure when it exists.
In the nonautonomous case, the natural counterpart of the invariant measure is not a single measure, but a one-parameter family of probability measures $\{\mu_t: t\in I\}$, which are
called \emph{evolution system of invariant measures} in \cite{DPR1} and \emph{entrance laws at $-\infty$} in \cite{dynkin},
and are characterized by the following property:
\begin{eqnarray*}
\int_{\R^d}G(t,s)fd\mu_t=\int_{\Rd}fd\mu_s,\qquad\;\,s<t,\;\,f\in C_b(\Rd).
\end{eqnarray*}
Such a property allows to extend each operator $G(t,s)$, in a straightforward way, to a contraction, mapping $L^p(\Rd,\mu_s)$
into $L^p(\Rd,\mu_t)$ for any $s,t\in I$, with $s<t$, and any $p\in [1,+\infty)$.
In \cite{LorLunZam10} the asymptotic behaviour of $G(t,s)$ in these $L^p$-spaces has been studied in the
case when the coefficients of the operator ${\mathcal A}$ are $T$-periodic with respect to the variable $t$. More precisely,
sufficient conditions guaranteeing that $\|G(t,s)f-m_s(f)\|_{L^p(\Rd,\mu_t)}$ goes to $0$ as $t-s\to +\infty$,
when $f\in L^p(\Rd,\mu_s)$ and $m_s(f)=\int_{\Rd}fd\mu_s$, have been obtained, thus generalizing the well-known convergence results of
the autonomous case.

In this paper, we are interested in studying nonautonomous elliptic operators with a nonzero potential term in $C_b(\Rd)$, i.e.,
we are interested in operators of the form
\begin{align*}
(\mathcal{A}(t)\psi)(x)&=\sum_{i,j=1}^d q_{ij}(t,x)D_{ij}\psi(x)+
\sum_{j=1}^d b_j(t,x)D_j\psi(x)-c(t,x)\psi(x),
\end{align*}
for any $(t,x)\in I\times\R^d$,
where $c$ is bounded from below and $I$ is as above.
Adapting the arguments used in the case of no potential term, we first show in Section \ref{sect-2} that we can
associate an evolution operator $(G(t,s))$ in $C_b(\Rd)$ with the operator ${\mathcal A}$. In fact, $G(t,s)f$ can be obtained as the ``limit'' as $n\to +\infty$, in an appropriate sense, of both the sequences of solutions to the Cauchy-Dirichlet and Cauchy-Neumann problems, for the equation $D_tu-{\mathcal A}(t)u=0$ in the ball $B(0,n)$.
Next, in Section \ref{sect-3} we show that it is possible to associate a Green function $g$ with the evolution operator $G(t,s)$, namely,
\begin{equation}
(G(t,s)f)(x)=\int_{\Rd}g(t,s,x,y)f(y)dy,\qquad\;\,s,t\in I,\;\,s<t,\;\,x\in\Rd,
\label{kernel-intro}
\end{equation}
for every $f\in C_b(\Rd)$.
For any fixed $s$ and almost any $y\in\Rd$, $g(\cdot,s,\cdot,y)$ is smooth and solves the equation $D_tg-\mathcal{A}(t)g=0$.
Formula \eqref{kernel-intro} allows us to extend each operator $G(t,s)$ to the space $B_b(\Rd)$ of all bounded and Borel measurable functions $f:\Rd\to\R$. The so extended operators turn out to be strong Feller (i.e., $G(t,s)$ maps $B_b(\Rd)$ into $C_b(\Rd)$) and
irreducible (i.e., if $U\neq\varnothing$ is a Borel measurable set, then $(G(t,s)\chi_{U})(x)>0$ for any $x\in\Rd$ and any $s<t$).
We also prove that, for any continuous function $f:\Rd\to\R$ vanishing at infinity and any $t\in I$, the function $G(t,\cdot)f$ is continuous
in $(-\infty,t]$ with values in $C_b(\Rd)$, for any $t\in I$. We then deduce that
$G(\cdot,\cdot)f$ is continuous in $\{(t,s,x)\in I\times I\times\Rd: t\ge s\}$.
Finally,  under an additional assumption, we establish an integral inequality which will play a crucial role in what follows.

Section \ref{sect-4} is devoted to the study of the compactness of the operator $G(t,s)$ in $C_b(\Rd)$. We show that
$G(t,s)$ is compact if and only if the family of measures $\{g(t,s,x,y)dy,~x\in\Rd\}$ (which are not probability measures if $c\neq 0$) is tight, where tightness means that for any $\varepsilon>0$ there exists $R_0>0$ such that
\begin{eqnarray*}
\sup_{x\in\Rd}\int_{\Rd\setminus B(0,R)}g(t,s,x,y)dy\le \varepsilon,
\end{eqnarray*}
provided that $R\geq R_0$. Sufficient conditions are then provided
for the previous family of measures be tight and consequently to show that, in this case, $G(t,s)$ preserves neither $C_0(\Rd)$ nor $L^p(\Rd)$.
Adapting some of the ideas in the proof of Theorem \ref{thm-comp}, we provide a sufficient condition
to guarantee that the function $G(\cdot,\cdot)f$ is continuous in $\{(t,s,x)\in I\times I\times\Rd: t\ge s\}$ for any
bounded and continuous function $f:\Rd\to\R$, thus extending the similar result of Section \ref{sect-3} proved
for functions vanishing at infinity.

Section \ref{sect-5} is then devoted to study the invariance of $C_0(\Rd)$ and $L^p(\Rd)$ ($p\in [1,+\infty)$), under the action
of $G(t,s)$, providing sufficient conditions for this property hold.
Examples of nonautonomous operators to which the main results of this paper apply are provided in Section \ref{sect-6}.

\subsection*{Notations}
We denote by $B_b(\R^d)$ the Banach space of all bounded and
Borel  measurable functions
$f:\R^d\to\R$, and by $C_b(\R^d)$ its subspace of all  continuous
functions.   $B_b(\R^d)$ and $C_b(\R^d)$ are endowed with the
sup norm $\|\cdot\|_{\infty}$. For $k>0$ $C^k_b(\R^d)$ is the
set of all functions $f\in C_b(\R^d)$ whose derivatives up to the
$[k]$th-order are bounded and $(k-[k])$-H\"older continuous in $\R^d$.
Here, $[k]$ denotes the integer part of $k$.
We use the
subscript ``$c$'' (resp. ``$0$'') instead of ``$b$''  for spaces of functions  with compact support
(resp. for spaces of functions vanishing at infinity).

Let ${\mathcal O}\subset\R^{1+d}$ be an open set  or the closure of
an open set. For $0<\alpha <1$ we denote by  $C^{\alpha/2,\alpha}_{\rm loc}({\mathcal O})$  the set of  functions $f:{\mathcal O}\to\R$
whose restrictions to any compact set ${\mathcal O}_0\subset {\mathcal O}$ belong to $C^{\alpha/2,\alpha} ({\mathcal O}_0)$.
Similarly, $C^{1+\alpha/2,2+\alpha}_{\rm loc}({\mathcal O})$ is the
subset of $C({\mathcal O})$ of the
functions $f$ such that  the time derivative $D_tf:=\frac{\partial f}{\partial t}$ and the spatial derivatives $D_if:=\frac{\partial
f}{\partial x_i}$, $D_{ij}f:=\frac{\partial^2f}{\partial x_i\partial x_j}$ exist and belong to $C^{\alpha/2,\alpha}_{\rm loc}({\mathcal O})$.

We denote by ${\rm Tr}(Q)$ and $\langle x,y\rangle$ the trace of the square matrix $Q$ and the Euclidean scalar product
of the vectors $x,y\in\Rd$, respectively.
By $\chi_A$ we denote the characteristic function of the set $A\subset\Rd$ and by $\one$ we denote the function which is identically
equal to 1 in $\Rd$.

We set $\Lambda:=\{(t,s)\in I\times I: t\geq s\}$ and, for every bounded set $J\subset I$, we denote by $\Lambda_J$ the intersection
of $\Lambda$ and $J\times J$. For any $t\in I$ we denote by $I_t$ the intersection of $I$ and $(-\infty,t]$ and
by ${\mathcal A}_{\lambda}$ the operator ${\mathcal A}-\lambda I$, for any $\lambda\in\R$.
Finally, by $a\vee b$ and $a\wedge b$ we denote, respectively, the maximum and the minimum between $a,b\in\R$.

\section{The evolution operator}
\label{sect-2}
Let $I$ be an interval, which is either $\R$ or a right halfline, and let the operators $\mathcal{A}(t)$, $t \in I$,
be defined on smooth functions $\psi$ by
\begin{align*}
(\mathcal{A}(t)\psi)(x)&=\sum_{i,j=1}^d q_{ij}(t,x)D_{ij}\psi(x)+
\sum_{i=1}^d b_i(t,x)D_i\psi(x)-c(t,x)\psi(x),
\end{align*}
for any $(t,x)\in I\times\R^d$, under the following hypothesis:
\begin{hyp}\label{hyp1}
\begin{enumerate}[\rm (i)]
\item
$q_{ij}$, $b_{i}$ $(i,j=1,\dots,d)$ and $c$ belong
to $C^{\alpha/2,\alpha}_{\rm loc}(I\times \R^d)$;
\item
$c_0:=\inf_{I\times\R^d}c>-\infty$;
\item
for every $(t,x)\in I\times \R^d$, the matrix $Q(t,x)=(q_{ij}(t,x))$
is symmetric and there exists a function $\eta:I\times\R^d\to\R$ such that $0<\eta_0:=\inf_{I\times\R^d}\eta$
and
\begin{eqnarray*}
\langle Q(t,x)\xi,\xi\rangle\geq\eta(t,x)|\xi|^2,\quad \xi\in \R^d,\quad (t,x)\in I\times \R^d;
\end{eqnarray*}
\item
for every bounded interval $J\subset I$ there exist a positive function $\varphi=\varphi_J\in C^2(\R^d)$
and a real number $\lambda=\lambda_J$ such that
\begin{eqnarray*}
\;\;\;\;\;\qquad\lim_{|x|\to +\infty}\varphi(x)=+\infty \quad\textrm{and}\quad
(\mathcal{A}(t)\varphi)(x)-\lambda\varphi(x)\leq 0,\quad (t,x)\in J\times \R^d.
\end{eqnarray*}
\end{enumerate}
\end{hyp}


We start by proving a maximum principle.
\begin{prop}\label{primax}
Let $s\in I$, $T>s$ and $R>0$. If $u \in C_b([s,T]\times
\Rd\setminus B(0,R))\cap C^{1,2}((s,T]\times
\Rd\setminus\overline{B(0,R)})$ satisfies
\begin{displaymath}
\left\{
\begin{array}{ll}
D_tu(t,x)-\mathcal{A}(t)u(t,x)\leq 0,\quad &(t,x)\in (s,T]\times \R^d\setminus\overline{B(0,R)},\\[1mm]
u(t,x)\le 0, &(t,x)\in [s,T]\times\partial B(0,R),\\[1mm]
u(s,x)\leq 0, &x\in \R^d,
\end{array}\right.
\end{displaymath}
then $u\leq 0$. Similarly, if $u\in C_b([s,T]\times \Rd)\cap C^{1,2}((s,T]\times \Rd)$
satisfies
\begin{displaymath}
\left\{
\begin{array}{ll}
D_tu(t,x)-\mathcal{A}(t)u(t,x)\leq 0,\quad &(t,x)\in (s,T]\times \R^d,\\[1mm]
u(s,x)\leq 0, &x\in \R^d,
\end{array}\right.
\end{displaymath}
then $u\le 0$. In particular,
if $u\in C_b([s,T]\times\R^d)\cap
C^{1,2}((s,T]\times\R^d)$ solves the Cauchy problem
\begin{displaymath}
\left\{
\begin{array}{ll}
D_tu(t,x)-\mathcal{A}(t)u(t,x)=0,\quad &(t,x)\in (s,T]\times \R^d,\\[1mm]
u(s,x)=f(x),  &x\in \R^d,
\end{array}\right.
\end{displaymath}
then
\begin{displaymath}
\|u(t,\cdot)\|_{\infty}\le e^{-c_0(t-s)}\|f\|_{\infty},\qquad\;\,t>s.
\end{displaymath}
\end{prop}

\begin{proof}
The proof is similar to that of the autonomous case. For the reader's convenience we go into details.

Without loss of generality we can assume that $\lambda>-c_0$. As it
is immediately seen, for any $n\in\mathbb N$, the function
$v_n(t,x)=e^{-\lambda (t-s)}u(t,x)-n^{-1}\varphi(x)$ satisfies the
inequalities
\begin{displaymath}
\left\{
\begin{array}{ll}
D_tv_n(t,x)-{\mathcal A}_{\lambda}(t)v_n(t,x)\leq 0,\quad &(t,x)\in (s,T]\times
\R^d\setminus\overline{B(0,R)},\\[1mm]
v_n(t,x)\le 0,\quad &(t,x)\in
[s,T]\times\partial B(0,R),\\[1mm]
 v_n(s,x)\leq 0, &x\in
\R^d.
\end{array}\right.
\end{displaymath}
Since $u$ is bounded in $[s,T]\times\R^d$ and $\varphi$ blows up as
$|x|\to +\infty$, the function $v_n$ tends to $-\infty$ as $|x|\to
+\infty$, uniformly with respect to $t\in [s,T]$. Hence, it has a
maximum at some point $(t_0,x_0)$. Such a maximum cannot be
positive, otherwise it would be $t_0>s$ and
$x_0\in\Rd\setminus\overline{B(0,R)}$, and from the differential
inequality we would be led to a contradiction. Hence, $v_n\le 0$ in
$[s,T]\times\Rd\setminus B(0,R)$. Letting $n\to+\infty$, yields
$u\le 0$ in $[s,T]\times\Rd\setminus B(0,R)$. Clearly, the same proof
can be applied to show the second statement of the theorem.

To prove the last part of the statement, it suffices to consider the
functions $v_{\pm}$ defined by $v_{\pm}(t,x)=\pm e^{c_0(t-s)}u(t,x)-\|f\|_{\infty}$ for any $(t,x)\in [s,T]\times\Rd$,
which satisfy the differential inequalities
\begin{displaymath}
\left\{
\begin{array}{ll}
D_tv_{\pm}(t,x)-\mathcal{A}_{-c_0}(t)v_{\pm}(t,x)\leq 0,\quad &(t,x)\in (s,T]\times \R^d,\\[1mm]
v_{\pm}(s,x)\leq 0,  &x\in \R^d.
\end{array}\right.
\end{displaymath}
The previous results, applied to the operator ${\mathcal A}_{-c_0}$ (which clearly satisfies Hypothesis \ref{hyp1}), show that $v_{\pm}(t,x)\le 0$ for any $(t,x)\in
[s,T]\times\R^d$ and this gives the assertion at once.
\end{proof}

We can now prove an existence-uniqueness result for the Cauchy problem
\begin{equation}\label{NonA}
\displaystyle{
\left\{
\begin{array}{ll}
D_tu(t,x)=\mathcal{A}(t)u(t,x),\quad &(t,x)\in (s,+\infty)\times \R^d,\\[1mm]
u(s,x)=f(x), & x\in \Rd,
\end{array}\right.
}
\end{equation}
with datum $f \in C_b(\Rd)$. For this purpose for any $n\in\mathbb N$ we introduce the Cauchy problems
\begin{equation}
\displaystyle{ \left\{
\begin{array}{ll}
D_tu_n(t,x)=\mathcal{A}(t)u_n(t,x),\quad &(t,x)\in (s,+\infty)\times B(0,n),\\[1mm]
u_n(t,x)=0, & (t,x)\in (s,+\infty)\times\partial B(0,n),\\[1mm]
u_n(s,x)=f(x), & x\in B(0,n)
\end{array}\right.
} \label{pb-approx-Dirichlet}
\end{equation}
and
\begin{equation}
\displaystyle{ \left\{
\begin{array}{ll}
D_tu_n(t,x)=\mathcal{A}(t)u_n(t,x),\quad &(t,x)\in (s,+\infty)\times B(0,n),\\[1mm]
\displaystyle\frac{\partial u_n}{\partial \nu}(t,x)=0, & (t,x)\in (s,+\infty)\times\partial B(0,n),\\[2mm]
u_n(s,x)=f(x), & x\in B(0,n),
\end{array}\right.
} \label{pb-approx-Neumann1}
\end{equation}
where $\nu=\nu(x)$ denotes the exterior unit normal at $x\in\partial B(0,n)$.
We further denote by $G_n^D(\cdot,s)$ and $G_n^N(\cdot,s)$ the bounded operators on $C_b(\Rd)$ which associate with any $f\in C_b(\Rd)$
the unique classical solution to problems \eqref{pb-approx-Dirichlet} and \eqref{pb-approx-Neumann1}, respectively.

\begin{thm}
\label{thm-1.3} For any $f\in
C_b(\R^d)$ and any $s\in I$ the Cauchy problem \eqref{NonA}
admits a unique solution $u_f\in C([s,+\infty)\times\R^d)\cap
C^{1+\alpha/2,2+\alpha}_{{\rm loc}}((s,+\infty)\times\R^d)$ $(\alpha$ being given by Hypothesis $\ref{hyp1}(i))$, which is bounded in $[s,T]\times\Rd$ for any $T>s$. For any $t>s$ and any $f\in C_b(\Rd)$, set $G(t,s)f:=u_f(t,\cdot)$.
Then, $G(t,s)$ is a bounded linear operator in $C_b(\R^d)$ and
\begin{equation}
\|G(t,s)\|_{{\mathcal L}(C_b(\R^d))}\le e^{-c_0(t-s)},\qquad\;\,t\geq s.
\label{stima-sol}
\end{equation}
Moreover, the following properties hold true:
\begin{enumerate}[\rm (i)]
\item
for any $f \in C_b(\R^d)$, $G_n^N(\cdot,s)f$ converges to $G(\cdot,s)f$
in $C^{1,2}(D)$ for
any compact set $D\subset (s,+\infty)\times\R^d$;
\item
for any $f\in C_b(\R^d)$ and any $s\in I$, the function $G_n^D(\cdot,s)f$ converges to $G(\cdot,s)f$ in $C^{1,2}(D)$ for
any compact set $D\subset (s,+\infty)\times\R^d$. Moreover,
if $f$ is nonnegative, then $G_n(t,s)f$ is increasing to $G(t,s)f$ for any $(t,s)\in\Lambda$.
\end{enumerate}
\end{thm}

\begin{proof}
Let us prove the first part of the statement and property (i). The uniqueness of the solution to problem \eqref{NonA} and
estimate \eqref{stima-sol} follow from Proposition \ref{primax}. Let us now prove that, for any $f\in C_b(\Rd)$, $G_n^N(\cdot,s)f$ converges, up to a subsequence, to
a solution to problem \eqref{NonA} which satisfies the properties in the statement of the theorem.
For this purpose we fix $f\in C_b(\R^d)$. The Schauder estimates in \cite[Thms. IV.5.3, IV.10.1]{LadSolUra68Lin}
show that the sequence
$\|G_n^N(\cdot,s)f\|_{C^{1+\alpha/2,2+\alpha}(K)}$ is bounded,
for any compact set $K\subset (s,T)\times\R^d$, by a constant
independent of $n$. Arzel\`a-Ascoli theorem, the arbitrariness of $K$ and a diagonal argument allow to conclude that there exists
a subsequence $(G_{n_k}^N(\cdot,s)f)$ which converges to a function $u\in C^{1+\alpha/2,2+\alpha}_{\rm loc}((s,+\infty)\times\Rd)$ in $C^{1,2}(D)$,
for any compact set $D\subset (s,+\infty)\times\R^d$. Clearly, $u$ satisfies the differential equation
in \eqref{NonA}.
Hence, to prove that $u$ solves problem \eqref{NonA} we just need to show that $u$ is continuous at $t=s$ and it therein equals the
function $f$. As a byproduct we also then deduce that the whole sequence $(G_n^N(\cdot,s)f)$ converges to $u$
in $C^{1,2}(D)$ for any compact set $D\subset (s,+\infty)\times\Rd$, since our arguments show that any subsequence of
$(G_n^N(\cdot,s)f)$ has a subsequence converging to $u$ in $C^{1,2}(D)$ for any $D$ as above.

Let us first suppose that $f$ belongs to $C^{2+\alpha}_c(\Rd)$. In this case we can estimate
$\|G_n^N(\cdot,s)f\|_{C^{1+\alpha/2,2+\alpha}(D)}$ from above by a constant, which is independent of $n$, for any compact set
$D\subset [s,+\infty)\times\R^d$ and any $n\in\mathbb N$ such that ${\rm supp}(f)\subset B(0,n)$.
Hence, $G_{n_k}^N(\cdot,s)f$ converges to $u$ uniformly in $D$ and, as a byproduct, $u$ is continuous up to
$t=s$ and is a solution to problem \eqref{NonA}.

Let us now assume that $f\in C_0(\Rd)$ and let $(f_m)\subset C^{2+\alpha}_c(\Rd)$ converge to $f$ uniformly
in $\Rd$. Then, using the classical maximum principle, which shows that $\|G^N_n(t,s)g\|_{\infty}\le e^{-c_0(t-s)}\|g\|_{\infty}$ for
any $g\in C(\overline{B(0,n)})$ and any $n\in\mathbb N$, we can estimate
\begin{align*}
|(G_{n_k}^N(t,s)f)(x)-f(x)|\le &|(G_{n_k}^N(t,s)f)(x)-(G_{n_k}^N(t,s)f_m)(x)|\\
&+|(G_{n_k}^N(t,s)f_m)(x)-f_m(x)|+|f_m(x)-f(x)|\\
\le &(e^{-c_0(t-s)}+1)\|f-f_m\|_{\infty}+|(G_{n_k}^N(t,s)f_m)(x)-f_m(x)|,
\end{align*}
for any $t>s$ and any $x\in\Rd$.
Letting $k\to +\infty$ yields
\begin{align*}
|(u(t,x)-f(x)|\le (e^{-c_0(t-s)}+1)\|f-f_m\|_{\infty}+|u_{f_m}(t,x)-f_m(x)|
\end{align*}
for any $(t,x)\in (s,+\infty)\times\Rd$ and any $m\in\mathbb N$, which clearly implies that
$u(t,\cdot)$ tends to $f$ as $t\to s^+$, locally uniformly in $\Rd$.

To conclude, let us consider the case when $f$ is merely bounded and continuous in $\R^d$.
Fix $R>0$ and let $\eta\in C^{2+\alpha}_c(\Rd)$ satisfy $\eta\equiv 1$ in $B(0,R)$ and $0\le\eta\le 1$ in $\Rd$. Further, let $(f_n)\subset C_c^{2+\alpha}(\R^d)$ be a bounded sequence with respect to the sup-norm converging to $f$ locally uniformly in $\R^d$, and set $M=\sup_{n\in \N}\|f_n\|_\infty$. Note that
\begin{eqnarray*}
|G_n^N(t,s)((\one-\eta)f_n)|\le \|f_n\|_{\infty}G_n^N(t,s)(\one-\eta)\le M(e^{-c_0(t-s)}-G_n^N(t,s)\eta),
\end{eqnarray*}
for any $s<t$ and any $n\in\mathbb N$,
as it follows immediately from the positivity of each operator $G_n^N(t,s)$.
Hence, we can estimate
\begin{align*}
&|G_{n_k}^N(t,s)f(x)-f(x)|\\
\le & |G_{n_k}^N(t,s)(f-f_n)(x)|
+|G_{n_k}^N(t,s)(f_n(1-\eta))(x)|\\
&+|G_{n_k}^N(t,s)(f_n\eta)(x)-(f_n\eta)(x)|\\
\le & e^{-c_0(t-s)}\|f-f_n\|_{L^{\infty}(B(0,n_k))}+M\left (e^{-c_0(t-s)}-(G_{n_k}^N(t,s)\eta)(x)\right )\\
&+|G_{n_k}^N(t,s)(f_n\eta)(x)-(f_n\eta)(x)|,
\end{align*}
for any $x\in B(0,R)$ and any $k$ such that $n_k>R$. Letting first $n\to +\infty$ and then $k\to +\infty$ we get
\begin{align*}
|u(t,x)-f(x)|\le
M\left (e^{-c_0(t-s)}-u_{\eta}(t,x)\right )+|u_{f\eta}(t,x)-(f\eta)(x)|,
\end{align*}
for any $x\in B(0,R)$. Letting $t\to s^+$, we see that $u(t,\cdot)\to f$, uniformly in $B(0,R)$.

The proof of property (ii) follows the same lines of the proof of property (i). Hence, we skip the details. We just observe that
the pointwise convergence of the sequence $(G_n^D(t,s)f)$ can also be proved applying the classical maximum principle to the function $G_m^D(\cdot,s)f-G_n^D(\cdot,s)f$ ($n,m\in\N$, $m>n$), which shows that, if $f\ge 0$, then
$G_m^D(\cdot,s)f-G_n^D(\cdot,s)f\ge 0$ in $[s,+\infty)\times B(0,n)$.
\end{proof}

\section{Basic properties of the operator $G(t,s)$}
\label{sect-3}
Let us now prove some properties of the operator $G(t,s)$. For this purpose, we set $G(t,t):=id_{C_b(\R^d)}$.

\begin{prop}[Green kernel]\label{greenkernel}
The following properties are satisfied.
\begin{enumerate}[\rm (i)]
\item
The family of operators $G(t,s)$ $(t,s\in I$, $s<t)$ defines an evolution operator on $C_b(\R^d)$.
\item
The evolution operator $(G(t,s))$ can be represented in the form
\begin{equation}
(G(t,s)f)(x)=\int_{\R^d}g(t,s,x,y)f(y)dy,\qquad\;\,s<t,\;\,x\in\R^d,
\label{repres-formula}
\end{equation}
for any $f \in C_b(\R^d)$, where $g:\Lambda\times\R^d\times\R^d\to\R$ is a positive function.
For any $s\in I$ and almost any $y\in \R^d$, $g(\cdot,s,\cdot,y)$ belongs to $C^{1+\alpha/2,2+\alpha}_{\rm loc}
((s,+\infty)\times \R^d)$ and solves the equation
$D_t g -\mathcal{A}(t)g=0$ in $(s,+\infty)\times \R^d$. Moreover,
\begin{equation}\label{1_est}
\|g(t,s,x,\cdot)\|_{L^1(\R^d)}\le e^{-c_0(t-s)},\qquad\;\,s<t,\;\,x\in\R^d.
\end{equation}
The function $g$ is called the Green function of $D_t u-\mathcal{A}(t)u=0$ in $(s,+\infty)\times \R^d$.
\item
$G(t,s)$ can be extended to $B_b(\Rd)$ through formula
\eqref{repres-formula}. Each operator $G(t,s)$ is irreducible and has the strong Feller property.
\end{enumerate}
\end{prop}

\begin{proof}
(i). It follows from the uniqueness of the solution to problem \eqref{NonA}. Indeed, for any $r<s$ and any $f\in C_b(\R^d)$, the function
$G(\cdot,r)f$ belongs to $C([s,+\infty)\times\R^d)\cap C^{1+\alpha/2,2+\alpha}_{\rm loc}((s,+\infty)\times\R^d)$, is bounded in $[s,T]\times\Rd$ for any
$T>s$ and solves the Cauchy problem
\begin{eqnarray*}
\displaystyle{
\left\{
\begin{array}{ll}
D_tu(t,x)=\mathcal{A}(t)u(t,x),\quad &(t,x)\in (s,+\infty)\times \R^d,\\[1mm]
u(s,x)=(G(s,r)f)(x), & x\in \Rd.
\end{array}\right.
}
\end{eqnarray*}
Hence, by uniqueness, $G(t,r)f=G(t,s)G(s,r)f$ for any $t>s$.

(ii). By \cite[Theorem 3.7.16]{Fri64Par} we know that for every $n \in \N$ there exists a unique Green function $g_n$ of the Cauchy-Dirichlet
problem \eqref{pb-approx-Dirichlet} in $(s,+\infty)\times B(0,n)$, i.e., a unique function $g_n$ such that
\begin{equation*}
(G_n^D(t,s)f)(x)=\int_{B(0,n)}g_n(t,s,x,y)f(y)dy,\qquad\;t>s,\;\,x\in B(0,n),
\end{equation*}
for any $f \in C(\overline{B(0,n)})$.
The function $g_n$ is positive and, as a function of $(t,x)$, it belongs to
$C^{1+\alpha/2,2+\alpha}((\tau,T)\times B(0,n))$ for every fixed $y\in B(0,n)$, $s\in I$ and $s<\tau<T$. Moreover, it satisfies
$D_t g_n-\mathcal{A}(t)g_n=0$ in $(s,+\infty)\times B(0,n)$. By Theorem \ref{thm-1.3}(ii),
for any nonnegative $f\in C_b(\R^d)$, the sequence $((G_n^D(t,s)f)(x))$ increases to $(G(t,s)f)(x)$. As a byproduct, the functions $g_n$
increase with $n$. Therefore, defining
\begin{eqnarray*}
g(t,s,x,y)=\lim_{n\to +\infty}g_n(t,s,x,y),\quad\,(t,s,x,y)\in \Lambda\times\R^d\times\R^d,
\end{eqnarray*}
by monotone convergence we get that
\begin{align*}
(G(t,s)f)(x)&=\lim_{n \to +\infty}(G_n^D(t,s)f)(x)=\int_{\R^d} g(t,s,x,y)f(y)\,dy,
\end{align*}
for any $f\ge 0$. For a general $f\in C_b(\Rd)$ it suffices to split $f=f^+-f^-$ and to apply the above argument to $f^+$ and $f^-$.
This shows that \eqref{repres-formula} holds.
The positivity of $g$ is obvious since each function $g_n$ is positive in $\Lambda\times B(0,n)\times B(0,n)$.
By \eqref{stima-sol} we have that
$$\int_{\R^d}g(t,s,x,y)\,dy=(G(t,s)\one)(x)\leq e^{-c_0(t-s)},\quad\quad t\geq s,\;\, x\in \R^d,$$
and \eqref{1_est} is proved.

As far as the regularity of $g$ with respect to the variables $t,x$ is concerned, we first show that,
for every $s \in I$ and almost all $y \in \R^d$, $g_n(\cdot,s,\cdot,y)$
is locally bounded in $I\times\Rd$, uniformly with respect to $n$.
Once this property is checked, the same argument used in the proof of Theorem \ref{thm-1.3}(i) based on interior Schauder estimates and Arzel\`a-Ascoli theorem, will show that $g(\cdot,s,\cdot,y)$ belongs to $C^{1+\alpha/2,2+\alpha}_{\rm loc}(I\times\Rd)$ for every $s \in I$ and almost any $y\in\R^d$.
So, let us fix two compact sets $[\tau,T]\subset (s,+\infty)$
and $K\subset\Rd$. Further, denote by $(t_h)$ and $(x_k)$ two countable sets dense in $[\tau,T+1]$ and in $K$, respectively.
Since $\int_{\R^d}g(t_h,s,x_k,y)\,dy<+\infty$ for any $h,k\in\mathbb N$, there exists a set ${\mathcal Y}\subset\Rd$ with negligible complement such that $g(t_h,s,x_k,y)<+\infty$ for any $y \in \mathcal{Y}$ and any $h,k\in\mathbb N$.
Let $\bar{y} \in {\mathcal Y}$ and let $R$ be sufficiently large such that $s<\tau-2/R$ and $\bigcup_{x\in K}\overline{B(x,1)}\subset B(0,R)$. Moreover, let $\vartheta$ be a smooth function compactly supported in $[\tau-2/R,T+2]\times\overline{B(0,R+1)}$ such that $0\le\vartheta\le 1$ and $\vartheta(t,x)=1$ for any $(t,x)\in [\tau-1/R,T+1]\times B(0,R)$. Define the operator $\tilde {\mathcal A}$ by setting
\begin{eqnarray*}
\tilde{\mathcal A}(t)\psi(x)=\sum_{i,j=1}^d\tilde q_{ij}(t,x)D_{ij}\psi(x)+\sum_{j=1}^d\tilde b_j(t,x)D_j\psi(x)-
\tilde c(t,x)\psi(x),
\end{eqnarray*}
where
\begin{align*}
&\tilde q_{ij}(t,x)=\vartheta(t,x)q_{ij}(t,x)+(1-\vartheta(t,x))\delta_{ij},\\
&\tilde b_j(t,x)=\vartheta(t,x)b_j(t,x),\\
&\tilde c(t,x)=\vartheta(t,x)c(t,x),
\end{align*}
for any $(t,x)\in \R^{d+1}$ and any $i,j=1,\ldots,d$. Since the function $g_n(\cdot,s,\cdot,y)$ satisfies the equation
$D_tg_n(\cdot,s,\cdot,y)-\tilde {\mathcal A}g_n(\cdot,s,\cdot,y)=0$ in $[\tau-1/R,T+1]\times B(0,R)$, for any $n>R$,
applying the Harnack inequality in \cite[Theorem 1]{Stu94Har},
we see that, if $\rho^2< 1\wedge 1/R$,
then there exists a positive constant $M_0$, independent of $h$, $k$ and $n$,  such that
\begin{equation}\label{Harnack}
g_n(t,s,x,\bar{y})\leq M_0 g_n(t_h,s,x_k,\bar{y})\le M_0 g(t_h,s,x_k,\bar{y}),
\end{equation}
for every $t \in [t_h-\frac{3}{4}\rho^2, t_h-\frac{1}{2}\rho^2]$, $x \in \overline{B(x_k,\rho/2)}$.
Since $[\tau,T]\times K$ can be covered by a finite number of cylinders $[t_h-\frac{3}{4}\rho^2, t_h-\frac{1}{2}\rho^2]\times B(x_k,\rho/2)$,
from \eqref{Harnack} we deduce that $g_n(\cdot,s,\cdot,\bar{y})$ is uniformly bounded in
$[\tau,T]\times K$ by a constant independent of $n$, as it has been claimed.

(iii). Clearly, the operator $G(t,s)$ can be extended to the set of all bounded Borel measurable functions
$f$ through formula \eqref{repres-formula}. To prove that $G(t,s)$ is strong Feller, we have to show that, for any
$f\in B_b(\Rd)$, $G(t,s)f$ is continuous. In fact, we will show that $G(t,s)f\in C^{2+\alpha}_{\rm loc}(\R^d)$.
For this purpose, we fix a bounded sequence $(f_n)$ of bounded and continuous functions converging pointwise to $f$ as
$n\to +\infty$. Clearly, $G(t,s)f_n$ converges to $G(t,s)f$ pointwise in $\R^d$ by dominated convergence.
Using the Schauder interior estimates, one can easily deduce that, for any $R>0$, the sequence $(G(t,s)f_n)$ is bounded in $C^{2+\alpha}(B(0,R))$.
Hence, Arzel\`a-Ascoli theorem implies that $G(t,s)f_n$ converges in $C^{2}(B(0,R))$ to $G(t,s)f$ and $G(t,s)f$ belongs to $C^{2+\alpha}(B(0,R))$.
This completes the proof.
\end{proof}


The following corollary is an immediate consequence of Proposition \ref{greenkernel}. Hence, we skip the proof.

\begin{coro}
For every $(t,s)\in \Lambda$ and every $x \in \R^d$ let us
define the measure $g_{t,s}(x,dy)$ by setting $g_{t,t}(x,dy)=\delta_x$ and
\begin{equation}\label{kernel}
g_{t,s}(x,A)=\int_A g(t,s,x,y)\,dy,\qquad\;\,t>s,
\end{equation}
for any Borel set $A\subset\Rd$. Then, each measure $g_{t,s}(x,dy)$ is equivalent to the Lebesgue measure $($i.e., it
has the same sets with zero measure as the restriction of the Lebesgue measure to the $\sigma$-algebra of
all the Borel sets of $\Rd)$. Moreover, for any $t\geq r\geq s$, $x\in \R^d$ and any Borel set $A\subset\R^d$ it holds that
\begin{eqnarray*}
g_{t,s}(x,A)=\int_{\R^d}g_{r,s}(y,A)g_{t,r}(x,dy).
\end{eqnarray*}
\end{coro}

The following lemma besides showing some continuity properties of the function $s\mapsto (G(t,s)f)(x)$ will be the key tool
to prove the compactness of the operator $G(t,s)$ in Theorem \ref{thm-comp}. Let us consider the
following additional hypothesis.
\begin{hyp}\label{hyp1-bis}
For every bounded interval $J\subset I$ there exist a positive function
$\varphi=\varphi_J\in C^2(\R^d)$ and a real number
$\lambda=\lambda_J$ such that
\begin{eqnarray*}
\lim_{|x|\to +\infty}\varphi(x)=+\infty \quad\textrm{and}\quad
(\mathcal{A}_{-c}(t)\varphi)(x)-\lambda\varphi(x)\leq 0,\quad (t,x)\in
J\times \R^d,
\end{eqnarray*}
where $\mathcal{A}_{-c}=\mathcal{A}+c I$.
\end{hyp}

\begin{lemm}\label{lemma_luca_luciana}
The following properties hold true:
\begin{enumerate}[\rm (i)]
\item
Suppose that $f\in C^2_c(\R^d)$. Then,
\begin{equation}
(G(t,s_1)f)(x)-(G(t,s_0)f)(x)=-\int_{s_0}^{s_1}(G(t,\sigma)\mathcal{A}(\sigma)f)(x)d\sigma,
\label{form-fond}
\end{equation}
for any $s_0\le s_1\le t$ and any $x\in\R^d$.
In particular, the function $(G(t,\cdot)f)(x)$ is differentiable in
$I_t$ for any $x\in\R^d$ and
\begin{eqnarray*}
\frac{\partial}{\partial s}(G(t,s)f)(x)=-(G(t,s)\mathcal{A}(s)f)(x).
\end{eqnarray*}
\item
Let $c\ge 0$, $f\in C^2_b(\R^d)$ be constant and positive outside a ball and
assume Hypothesis $\ref{hyp1-bis}$. Then, for any $x\in\R^d$, the
function $(G(t,\cdot)\mathcal{A}(\cdot)f)(x)$ is locally integrable in
$I_t$ and
\begin{eqnarray*}
(G(t,s_1)f)(x)-(G(t,s_0)f)(x)\ge-\int_{s_0}^{s_1}(G(t,\sigma)\mathcal{A}(\sigma)f)(x)d\sigma,
\end{eqnarray*}
for any $s_0\le s_1\le t$.
\end{enumerate}
\end{lemm}

\begin{proof}
(i). Let us fix $f\in C^2_c(\R^d)$ and let $n$ be sufficiently large
such that ${\rm supp}(f)\subset B(0,n)$. By \cite[Theorem 2.3(ix)]{Acq88Evo}
\begin{equation}
(G_n^D(t,s_1)f)(x)-(G_n^D(t,s_0)f)(x)=-\int_{s_0}^{s_1}(G_n^D(t,r)\mathcal{A}(r)f)(x)dr,
\label{form-fond-n}
\end{equation}
for any $s_0\le s_1\le t$ and any $x\in\R^d$,
where we recall that $(G_n^D(t,s))$ is the evolution operator associated with the Cauchy-Dirichlet problem
\eqref{pb-approx-Dirichlet}.
Since the function $(r,x)\mapsto (\mathcal{A}(r)f)(x)$ is bounded and continuous in $[s_0,s_1]\times\R^d$, taking Theorem \ref{thm-1.3}(ii)
into account, we can let $n\to +\infty$ in \eqref{form-fond-n} and obtain \eqref{form-fond}.

(ii). Since any function which is constant and positive in a
neighborhood of $\infty$ can be split into the sum of a compactly
supported function and a positive constant, due to the
above result we just need to consider the case when $f=\one$.

Being rather long, we split the proof into three steps.
To lighten the notation, throughout the proof we denote by
$\|\psi\|_{\infty,R}$ the sup-norm over the ball $B(0,R)$ of
the continuous function $\psi:\R^d\to\R$.

\emph{Step 1.} We first assume that the potential $c$
tends to $0$ as $|x|\to +\infty$, uniformly with respect to $t$
in bounded sets of $I$. As usual, let $(G_n^N(t,s))$ be the evolution operator
associated with the Cauchy-Neumann problem \eqref{pb-approx-Neumann1}.
As it is well known,
\begin{eqnarray*}
(G_n^N(t,s_2)f)(x)-(G_n^N(t,s_1)f)(x)=-\int_{s_1}^{s_2}(G_n^N(t,\tau){\mathcal
A}(\tau)f)(x)d\tau,
\end{eqnarray*}
for any $f\in C_b^2(\R^d)$ such that $\frac{\partial
f}{\partial\nu}=0$ on $\partial B(0,n)$, any $s_1,s_2\in I$ such that $s_1\le s_2\le t$ and any
$x\in B(0,n)$. In particular, taking $f=\one$ yields
\begin{equation}
(G_n^N(t,s_2)\one)(x)-(G_n^N(t,s_1)\one)(x)=\int_{s_1}^{s_2}(G_n^N(t,\tau)c(\tau,\cdot))(x)d\tau,
\label{algeri-0}
\end{equation}
for any $s_1$, $s_2$, $t$ and $x$ as above.
Theorem \ref{thm-1.3}(i) shows that
$(G_n^N(t,s_2)\one)(x)-(G_n^N(t,s_1)\one)(x)$ and $(G_n^N(t,\tau)c(\tau,\cdot))(x)$ tend to $(G(t,s_2)\one)(x)-(G(t,s_1)\one)(x)$
and $(G(t,\tau)c(\tau,\cdot))(x)$, respectively, as $n\to +\infty$. Hence,
taking the limit as $n\to +\infty$ in both the sides of \eqref{algeri-0}
yields, by dominated convergence,
\begin{equation}
(G(t,s_2)\one)(x)-(G(t,s_1)\one)(x)=\int_{s_1}^{s_2}(G(t,\tau)c(\tau,\cdot))(x)d\tau,\quad\;s_1\le
s_2<t,\;\,x\in\R^d. \label{algeri}
\end{equation}

{\em Step 2.} Let us now suppose that $c$ is unbounded.
Let us set $c_n(s,x)=c(s,x)\vartheta_n(x)$ for any $(s,x)\in I\times\Rd$,
where $\vartheta_n\in C_c(\R^d)$ satisfies $\chi_{B(0,n)}\le\vartheta_n\le\chi_{B(0,n+1)}$ for any $n\in\mathbb N$.
Clearly, each function $c_n$ is nonnegative and belongs to
$C(I;C_c(\R^d))$. Moreover,
$c_n(s,x)\le c(s,x)$ for any $(s,x)\in I\times\R^d$ and any $n\in\N$, and
the sequence $(c_n(s,x))$ is increasing for any $(s,x)\in I\times\R^d$.

By Hypothesis \ref{hyp1-bis}, each operator ${\mathcal A}_n$
satisfies the assumptions of Theorem \ref{thm-1.3}. Hence, we
can associate an evolution operator $G_n(t,s)$ with the
operator
\begin{eqnarray*}
{\mathscr
A}_n(t)\varphi(x)=\sum_{i,j=1}^dq_{ij}(t,x)D_{ij}\varphi(x)+\sum_{j=1}^db_j(t,x)D_j\varphi(x)
-c_n(t,x)\varphi(x),
\end{eqnarray*}
for any $(t,x)\in I\times\Rd$.
Note that, for any nonnegative $f\in C_b(\R^d)$ and any $m,n\in\mathbb N$ such that $n<m$, the function
$u=G_m(\cdot,s)f-G_n(\cdot,s)f$ satisfies the differential
inequality $D_tu-{\mathscr A}_nu\le 0$
and vanishes at $t=s$. The maximum principle in Proposition
\ref{primax} then implies that $u\le 0$ in $[s,+\infty)\times\R^d$, i.e.,
\begin{eqnarray*}
(G_m(t,s)f)(x)\le (G_n(t,s)f)(x),\qquad\;\,s\le
t,\;\,x\in\R^d.
\end{eqnarray*}
In particular, for any fixed $t>s$ and $x\in\Rd$, the sequence
$((G_n(t,s)f)(x))$ is nonincreasing. Hence, it converges to some function $u$ as $n\to +\infty$.
To show that $u=G(\cdot,s)f$, it suffices to use the
same arguments as in the proof of Theorem \ref{thm-1.3}. We leave the details to the reader.

\emph{Step 3.} We now complete the proof. Writing \eqref{algeri}
with $G_n$ replacing $G$, we get
\begin{align*}
(G_n(t,s_2)\one)(x)-(G_n(t,s_1)\one)(x)=&\int_{s_1}^{s_2}(G_n(t,\tau)c_n(\tau,\cdot))(x)d\tau\\
\ge &
\int_{s_1}^{s_2}(G(t,\tau)c_n(\tau,\cdot))(x)d\tau,
\end{align*}
for any $s_1,s_2\in I$ such that $s_1\le s_2\le t$ and any $x\in\R^d$. Since the sequence $(c_n)$
is increasing and $G(t, \tau)$ is positive, we can apply Fatou lemma to pass to the limit as
$n\to +\infty$ and get
\begin{eqnarray*}
(G(t,s_2)\one)(x)-(G(t,s_1)\one)(x)\ge\int_{s_1}^{s_2}(G(t,\tau)c(\tau,\cdot))(x)d\tau,
\end{eqnarray*}
for any $s_1$, $s_2$, $t$ and $x$ as above. This completes the
proof.
\end{proof}

\begin{coro}\label{cont-c_0}
For any $f\in C_0(\R^d)$ the following properties are satisfied:
\begin{enumerate}[\rm (i)]
\item
the function $G(t,\cdot)f$ belongs to $C(I_t;C_b(\Rd))$ for any $t\in I$;
\item
the function
$(t,s,x)\mapsto (G(t,s)f)(x)$ is continuous in $\Lambda\times \R^d$.
\end{enumerate}
\end{coro}
\begin{proof}
It suffices to prove the statements when $f\in C^\infty_c(\R^d)$. Indeed, the case when
$f \in C_0(\R^d)$ follows by density approximating $f$
uniformly in $\R^d$ by a sequence of functions $f_n\in C^\infty_c(\R^d)$
and taking into account that $G(\cdot,\cdot)f_n$ converges to $G(\cdot,\cdot)f$
uniformly in $K\times\R^d$ for any compact set $K\subset \Lambda$.

(i). Formula \eqref{form-fond} shows that
\begin{align*}
\|G(t,s_1)f-G(t,s_0)f\|_{\infty}&\le \sup_{r\in
[s_0,s_1]}\|G(t,r)\mathcal{A}(r)f\|_{\infty}|s_1-s_0|\\
&\le \sup_{r\in
[s_0,s_1]}(e^{-c_0(t-r)}\|\mathcal{A}(r)f\|_{\infty})|s_1-s_0|,
\end{align*}
for any $t\le s_0<s_1$, and this implies that the function $G(t,\cdot)f$ is locally Lipschitz continuous in
$(-\infty,t]$ with values in $C_b(\R^d)$.

(ii). Using the classical Schauder estimates in \cite[Theorem 3.5]{Fri64Par}, we can show that, for any compact set $[a,b]\subset I$, any
$m\in\mathbb N$ and any compact set $K\subset\Rd$, $\|G(\cdot,s)f\|_{C^{1+\alpha/2,2+\alpha}([s,s+m]\times K)}$ is bounded from above by a constant
$C_1$ independent of $s\in [a,b]$. In particular, this shows that
\begin{equation}
|(G(t_2,s)f)(x)-(G(t_1,s)f)(x_0)|\le C_1\left (|t_2-t_1|+|x-x_0|\right ),
\label{lipschitz-luglio}
\end{equation}
for any $t_1,t_2\in [s,s+m]$, any $x,x_0\in K$ and any $s\in [a,b]$.

Let $(t,s,x),\,(t_0,s_0,x_0) \in \Lambda\times \R^d$ with $s, s_0 \in [a,b]$ for some $[a,b]\subset I$. Assume that $s<s_0$; by \eqref{form-fond}
and \eqref{lipschitz-luglio} we can estimate
\begin{align*}
|(G(t,s)f)(x)-(G(t_0,s_0)f)(x_0)|&\leq |(G(t,s)f)(x)-(G(t_0,s)f)(x_0)|\nnm\\
&\quad+|(G(t_0,s)f)(x_0)-(G(t_0,s_0)f)(x_0)|\nnm\\
& \leq C_1\left (|t-t_0|+|x-x_0|\right )+ C_2 |s-s_0|,
\end{align*}
where $C_2=\sup_{r\in [a,b]}(e^{-c_0(t_0-r)}\|\mathcal{A}(r)f\|_{\infty})$.
Hence,
\begin{align*}
\lim_{(t,s,x)\to(t_0,s_0^-,x_0)}(G(t,s)f)(x)=(G(t_0,s_0)f)(x_0).
\end{align*}
Now, suppose that $s\geq s_0$ and $|t-t_0|\le 1$. Then, $(t,s_0)\in \Lambda$ and
\begin{align*}
|(G(t,s)f)(x)-(G(t_0,s_0)f)(x_0)|&\leq |(G(t,s)f)(x)-(G(t,s_0)f)(x)|\\
&\quad+|(G(t,s_0)f)(x)-(G(t_0,s_0)f)(x_0)|\\
& \leq  C_3 |s-s_0| + |(G(t,s_0)f)(x)-(G(t_0,s_0)f)(x_0)|,
\end{align*}
where $C_3:=\max_{t\in [t_0-1,t_0+1]}\max_{s\in [a,b]}(e^{-c_0(t-r)}\|\mathcal{A}(r)f\|_{\infty})$.
Hence,
\begin{align*}
\lim_{(t,s,x)\to(t_0,s_0^+,x_0)}(G(t,s)f)(x)=(G(t_0,s_0)f)(x_0),
\end{align*}
and the proof is completed.
\end{proof}

\section{Compactness of the evolution operator in $C_b(\R^d)$}
\label{sect-4}
We now give sufficient conditions ensuring that the operator $G(t,s)$ is compact.
We stress that in the case when $c\equiv 0$ (i.e., in the conservative case) a sufficient condition for
$G(t,s)$ be compact in $C_b(\Rd)$ has been established in \cite[Theorem 3.3]{Lun10Com}. For notational convenience, for any
interval $J\subset I$, we set
\begin{eqnarray*}
\widetilde\Lambda_J:=\{(t,s)\in J\times J: t>s\}.
\end{eqnarray*}

\begin{prop}\label{tight_equiv_comp}
Let $J\subset I$ be an interval.
The following properties are equivalent.
\begin{enumerate}[\rm (i)]
\item
for every $(t,s)\in \widetilde\Lambda_J$, $G(t,s)$ is compact in $C_b(\Rd)$;
\item
for every $(t,s)\in\widetilde\Lambda_J$ and every $\varepsilon>0$, there exists $R>0$ such that
\begin{equation}\label{tight}
g_{t,s}(x,\R^d\setminus B(0,R))\le\varepsilon,\qquad\;\,x\in\Rd,
\end{equation}
where the measures $g_{t,s}(x,dy)$ are defined in \eqref{kernel}.
\end{enumerate}
\end{prop}

\begin{proof}
$(i)\Rightarrow (ii)$.
Suppose that $G(t,s)$ is compact and consider a sequence $(f_n)$
such that $\chi_{\R^d\setminus B(0,n+1)}\le f_n\le \chi_{\R^d\setminus B(0,n)}$ for any $n\in\mathbb N$.
Clearly, $f_n$ converges to $0$ locally uniformly in $\R^d$, as $n\to +\infty$. Using the representation formula \eqref{repres-formula} it is
easy to check that $G(t,s)f_n$ converges to $0$ pointwise in $\R^d$.
Since the operator $G(t,s)$ is compact and the sequence $(f_n)$ is bounded, we can extract a subsequence
$(f_{n_k})$ such that $G(t,s)f_{n_k}$ converges to $0$ uniformly in $\R^d$. This is enough to infer that
the whole sequence $(G(t,s)f_n)$ tends to $0$, uniformly in $\R^d$, as $n\to +\infty$.

To complete the proof, it suffices to observe that
\begin{align*}
g_{t,s}(x,\R^d\setminus B(0,n))=(G(t,s)\chi_{\Rd\setminus B(0,n)})(x)\le
(G(t,s)f_{n-1})(x),
\end{align*}
for any $x\in\R^d$ and any $n\in\mathbb N$.
\medskip

$(ii)\Rightarrow (i)$. Fix $s,t\in I$ with $s<t$ and $r\in (s,t)$. Further, consider the family of operators $S_n$ ($n\in\mathbb N$) defined as follows:
\begin{eqnarray*}
S_nf=G(t,r)(\chi_{B(0,n)}G(r,s)f),\qquad\;\,f\in C_b(\R^d),\;\,n\in\mathbb N.
\end{eqnarray*}
Since $G(t,r)$ is strong Feller (see Proposition
\eqref{greenkernel}(iii)), $S_n$ is a bounded operator in $C_b(\R^d)$. Moreover,
\begin{align}
|(G(t,s)f)(x)-(S_nf)(x)|=&\left |\int_{\R^d\setminus B(0,n)}(G(r,s)f)(y)g_{t,r}(x,dy)\right |\notag\\
\le &\|G(r,s)f\|_{\infty}g_{t,r}(x,\R^d\setminus B(0,n)),
\label{comp-1}
\end{align}
for any $x\in\Rd$, and the last side of \eqref{comp-1} vanishes as $n\to +\infty$, uniformly with respect to $x\in\R^d$.
Hence, to prove the assertion it suffices to show that each operator $S_n$ is compact. This follows observing that
the operator $G(r,s)$ is compact from $C_b(\R^d)$ into $C(\overline{B(0,n)})$ for any $n\in\mathbb N$.
Indeed, the interior Schauder estimates imply that, for any bounded family ${\mathcal F}\subset C_b(\R^d)$, the family
$\mathcal{G}:=\{(G(s,r)f)_{|B(0,n)}: f\in {\mathcal F}\}$ is bounded in $C^{2+\alpha}(B(0,n))$. Therefore, $\mathcal{G}$
is equicontinuous and equibounded in $C(\overline{B(0,n)})$ by Arzel\`a-Ascoli theorem, i.e.,
the operator $f\mapsto (G(s,r)f)_{|B(0,n)}$ is compact. Thus, $S_n$ is compact as well. Being limit of compact operators, $G(t,s)$ is compact.
\end{proof}

In the following theorem we obtain a lower bound estimate for $g_{t,s}(x,\R^d)$ for every $t>s$ and any $x\in \R^d$,
which is crucial for the proof of Theorem \ref{thm-comp}.

\begin{prop}\label{notC_0}
Assume that Hypothesis $\ref{hyp1-bis}$ holds. Let $J\subset I$ be an interval
and suppose that there exist a positive and bounded function $W\in
C^2(\R^d\setminus B(0,R))$, $\mu\in\R$ and $R>0$ such that
$\inf_{x\in\R^d\setminus B(0,R)}W(x)>0$ and
\begin{equation}\label{cond_W}
\mathcal{A}(t)W-\mu W\ge 0, \quad (t,x)\in J\times \R^d\setminus B(0,R).
\end{equation} Then, for any $s_0,T\in J$, such that $T>s_0$,
there exists a positive constant $C_{T,s_0}$ such
that
\begin{equation}\label{l_bound}
\int_{\R^d}g_{t,s}(x,dy)\ge C_{T,s_0},
\end{equation}
for any $s,t\in\R$, with $s_0\le s\le t\le T$ and any $x\in\R^d$.
\end{prop}

\begin{proof}
We first assume that $c\ge 0$ and introduce the function $v$ defined by $v(t,x)=e^{-\mu
(t-s_0)}(G(t,s_0)\one)(x)$ for any $t\geq s_0$ and $x\in\R^d$.
Since $G(t,s_0)\one$ is everywhere positive in $\R^d$,
the minimum of $v$ over $[s_0,T]\times \overline{B(0,R)}$ is a positive constant, which we denote by $\kappa$.

Let $z:[s_0,T]\times\R^d\to\R$ be defined by $z(t,x)=v(t,x)-\gamma
W(x)$ for any $(t,x)\in [s_0,T]\times\R^d$, where
$\gamma=\kappa/\sup_{x\in \Rd\setminus B(0,R)}W(x)$. Clearly,
$z$ belongs to $C_b([s_0,T]\times\R^d)\cap
C^{1,2}((s_0,T]\times\R^d)$ and solves the
following problem:
\begin{eqnarray*}
\left\{
\begin{array}{lll}
D_tz(t,x)\ge\mathcal{A}_\mu(t)z(t,x), & t\in (s_0,T], &
x\in\R^d\setminus\overline{B(0,R)},\\[1mm]
z(t,x)\ge 0, &t\in [s_0,T], & x\in\partial B(0,R),\\[1mm]
z(s_0,x)\ge 0, &&x\in\R^d\setminus B(0,R).
\end{array}
\right.
\end{eqnarray*}
The maximum principle in Proposition
\ref{primax} implies that $z\ge 0$ in $[s_0,T]\times\R^d\setminus
B(0,R)$ or, equivalently,
\begin{eqnarray*}
e^{-\mu (t-s_0)}(G(t,s_0)\one)(x)\ge \gamma
W(x)\ge\gamma\inf_{y\in\R^d\setminus B(0,R)}W(y),
\end{eqnarray*}
for any $t\in
[s_0,T]$ and any $x\in\R^d\setminus B(0,R)$.
It thus follows that $G(t,s_0)\one\ge C_{s_0,T}$ in $\Rd$, for any
$s_0\le t\le T$, where
\begin{eqnarray*}
C_{s_0,T}=\min\{1,e^{\mu(T-s_0)}\}\min\left\{\kappa,\,\gamma\inf_{y\in\R^d\setminus B(0,R)}W(y)\right\}.
\end{eqnarray*}

Let us now fix $s$ such that $s_0<s<t$. From formula \eqref{algeri} we
infer that the function $(G(t,\cdot)\one)(x)$ is increasing. Therefore,
$(G(t,s)\one)(x)\ge (G(t,s_0)\one)(x)\ge C_{s_0,T}$ for any $x\in\R^d$,
and this accomplishes the proof in the case when $c\ge 0$, since by the representation formula
\eqref{repres-formula} and \eqref{kernel},
$g_{t,s}(x,\R^d)=(G(t,s)\one)(x)$ for any $s,t\in I$, with $s<t$, and any $x\in\Rd$.

In the general case when $c_0<0$, let
$(P(t,s))=(e^{c_0(t-s)}G(t,s))$ be the evolution operator associated with the second-order elliptic operator
\begin{eqnarray*}
{\mathcal A}_{-c_0}=\sum_{i,j=1}^dq_{ij}(t,x)D_{ij}+\sum_{j=1}^db_j(t,x)D_j-(c(t,x)-c_0).
\end{eqnarray*}
Clearly, the operator ${\mathcal A}_{-c_0}$ satisfies Hypotheses \ref{hyp1}(iv) and \ref{hyp1-bis} with the same $\lambda_J$ and $\varphi_J$.
Moreover, it fulfills also assumption \eqref{cond_W} with $\mu$ replaced with $c_0+\mu$.
Hence, from the above arguments, it follows that for any $s_0$, $T$ there exists a positive constant
$C_{s_0,T}'$ such that
$(P(t,s)1)(x)\ge C_{s_0,T}'$ for any $x\in\Rd$ and any $T\ge t\ge s \ge s_0$,
and \eqref{l_bound} follows with $C_{s_0,T}=C_{s_0,T}'$.
The proof is complete.
\end{proof}
Adapting to our situation the technique in \cite{MetPalWac02 Com},
we give a sufficient condition which ensures compactness of the family $G(t,s)$ for $t>s$
in the non conservative case.

\begin{thm}
\label{thm-comp} Assume that Hypothesis $\ref{hyp1-bis}$ is satisfied and there exist $R>0$, $d_1,d_2\in I$, with $d_1<d_2$, a
positive function $\varphi\in C^2(\Rd)$, blowing up as $|x|\to +\infty$, and a
convex increasing function $h:[0,+\infty)\to \R$ such that
$1/h\in L^1(a,+\infty)$ for large $a$ and
\begin{equation}
(\mathcal{A}(s)\varphi)(x)\leq -h(\varphi(x)),\qquad\;\, s\in [d_1,d_2],\;\, |x|\geq R.
\label{cond-comp}
\end{equation}
Finally, let the assumptions of Proposition $\ref{notC_0}$ hold true with $J=[d_1,d_2]$.
Then, $G(t,s)$ is compact in $C_b(\R^d)$ for any $(t,s)\in \{(t,s)\in \Lambda: s\le d_2,\,t\ge d_1,\, t\neq s\}$.
\end{thm}
\begin{proof}
Of course we can limit ourselves to proving the compactness of $G(t,s)$ for $(t,s)\in\widetilde\Lambda_{[d_1,d_2]}$ since
for the other values of $(t,s)$ it suffices to recall that $(G(t,s))$ is an evolution operator.

Let us first assume that $c\ge 0$. We will prove that the measures $g_{t,s}(x,dy)$ satisfy condition
\eqref{tight} for any $(t,s)\in\widetilde\Lambda_{[d_1,d_2]}$.
First of all we prove that the function $\varphi$ is integrable with respect to every measure
$g_{t,s}(x,dy)$ ($t>s$, $x\in\R^d$), so that $(G(t,s)\varphi)(x)$ is well defined for such $t$, $s$ and $x$.
For every $n\in \mathbb N$ choose $\psi_n \in C^{2}([0,+\infty))$ such that
\begin{itemize}
\item[(i)] $\psi_n(r)=r$ for $r\in [0,n]$,
\item[(ii)]$\psi_n(r)=n+\frac{1}{2}$ for $r\geq n+1$,
\item[(iii)]$0 \leq \psi_n'\leq 1$ and $\psi_n''\leq 0$.
\end{itemize}
Note that the previous conditions imply that $\psi_n'(r)r\leq\psi_n(r)$ for
every $r\in [0,+\infty)$. The function $\varphi_n:= \psi_n\circ\varphi$ belongs to
$C^2(\R^d)$  and is constant outside a compact set for any $n\in\N$. By
Lemma \ref{lemma_luca_luciana}(ii), the differential inequality
$\psi_n'(r)r\leq\psi_n(r)$ and the positivity of the function $G(t,s)\varphi$, we get
\begin{align}\label{est1}
\varphi_n(x)&\geq \varphi_n(x)-(G(t,s)\varphi_n)(x)\nnm\\
&\geq -\int_s^t (G(t,\sigma)\mathcal{A}(\sigma)\varphi_n)(x)d\sigma\nnm\\
& = -\int_s^t (G(t,\sigma)(\psi_n'(\varphi)\mathcal{A}(\sigma)\varphi+ \psi_n''(\varphi)\langle Q D\varphi,D\varphi\rangle + c(\psi_n'(\varphi)\varphi-\varphi_n))(x)d\sigma\nnm\\
& \geq -\int_s^t
(G(t,\sigma)(\psi_n'(\varphi)\mathcal{A}(\sigma)\varphi)(x)d\sigma.
\end{align}
The right-hand side of \eqref{est1} can be split
into two parts as follows:
\begin{align}\label{rhside}
\int_s^t (G(t,\sigma)(\psi_n'(\varphi)\mathcal{A}(\sigma)\varphi)(x)d\sigma =& \int_s^t d\sigma\int_{A_+(\sigma)}\psi_n'(\varphi(y))(\mathcal{A}(\sigma)\varphi)(y)g_{t,\sigma}(x,dy)\nnm\\
&+   \int_s^t d\sigma\int_{A_-(\sigma)}\psi_n'(\varphi(y))(\mathcal{A}(\sigma)\varphi)(y)g_{t,\sigma}(x,dy),
\end{align}
where $A_+(\sigma)=\{y\in\Rd: (\mathcal{A}(\sigma)\varphi)(y)>0\}$ and
$A_-(\sigma)=\{y\in\Rd: (\mathcal{A}(\sigma)\varphi)(y)\le 0\}$.
Since $(\psi_n'(\varphi))(y)$ is positive, increasing in $n$ and converges to $1$ for each $y\in \R^d$,
the first integral in the right-hand side of \eqref{rhside} converges by the monotone convergence theorem to $\int_s^t d\sigma\int_{A_+(\sigma)}(\mathcal{A}(\sigma)\varphi)(y)g_{t,\sigma}(x,dy)$ which is
finite since the sets $A_+(\sigma)$ are equibounded in $\R^d$ (note that $(\mathcal{A}(\sigma)\varphi)(x)$
tends to $-\infty$ as $|x|\to +\infty$ uniformly respect to $\sigma \in [d_1,d_2]$).
Now, using \eqref{est1} and \eqref{rhside} we get
\begin{align*}
&-\int_s^t d\sigma\int_{A_-(\sigma)}\psi_n'(\varphi(y))(\mathcal{A}(\sigma)\varphi)(y)g_{t,\sigma}(x,dy)\\
\le &\varphi_n(x)+\int_s^t
d\sigma\int_{A_+(\sigma)}\psi_n'(\varphi(y))(\mathcal{A}(\sigma)\varphi)(y)g_{t,\sigma}(x,dy).
\end{align*}
Letting $n \to + \infty$ we deduce that the integral $\int_s^t d\sigma\int_{A_-(\sigma)}(\mathcal{A}(\sigma)\varphi)(y)g_{t,\sigma}(x,dy)$
is finite as well as the integral $\int_s^t (G(t,\sigma)\mathcal{A}(\sigma)\varphi)(x)d\sigma$. Moreover, since
\begin{eqnarray*}
(G(t,s)\varphi_n)(x) \leq \int_s^t (G(t,\sigma)(\psi_n'(\varphi))\mathcal{A}(\sigma)\varphi)(x)d\sigma + \varphi_n(x),
\end{eqnarray*}
letting $n \to +\infty$ we also deduce that $(G(t,s)\varphi)(x)$ is finite for every $(t,s)\in \Lambda_{[d_1,d_2]}$ and any $x\in\R^d$.
Next, starting from the inequality
\begin{eqnarray*}
(G(t,s)\varphi_n)(x)-(G(t,r)\varphi_n)(x) \geq -\int_r^s (G(t,\sigma)\mathcal{A}(\sigma)\varphi_n)(x)d\sigma,\;\; r<s<t,\;\, x\in\R^d
\end{eqnarray*}
and arguing as above, we can show that
\begin{equation}\label{est_1}
(G(t,s)\varphi)(x)-(G(t,r)\varphi)(x) \geq -\int_r^s (G(t,\sigma)\mathcal{A}(\sigma)\varphi)(x)d\sigma,
\end{equation}
for every $r<s<t$ and $x\in\R^d$. Now, we prove that
$(G(t,s)\varphi)(x)$ is bounded by a constant independent of $x$.
Without loss of generality we can suppose
$(\mathcal{A}(s)\varphi)(x)\leq -h(\varphi(x))$, for any $s\in [d_1,d_2]$
and any $x\in \R^d$. Indeed, if this is not the case we replace $h$
by $h-C$ for a suitable constant $C$. We can also assume that $h$ vanishes at some point $x_h>0$.

 From the Jensen inequality for finite measures we get
\begin{eqnarray*}
h\left(\int_{\R^d}\varphi(y)g_{t,s}(x,dy)\right)\leq \frac{1}{g_{t,s}(x,\R^d)}\int_{\R^d}h(\varphi(y))g_{t,s}(x,dy)\quad t>s,\;\,x\in\R^d,
\end{eqnarray*}
since $0<g_{t,s}(x,\R^d)=(G(t,s)\one)(x)\leq 1$ for every $t\geq s$ and
$x\in\R^d$, and $h$ is increasing. We have thus obtained that
\begin{eqnarray*}
h((G(t,s)\varphi)(x))\leq
\frac{1}{g_{t,s}(x,\R^d)}(G(t,s)h(\varphi))(x),
\end{eqnarray*}
or, equivalently,
\begin{eqnarray*}
(G(t,s)h(\varphi))(x)\ge
g_{t,s}(x,\R^d)h((G(t,s)\varphi)(x)),\qquad\;\,t>s,\;\,x\in\R^d.
\end{eqnarray*}
Fix $s_0<T$. Then, by Proposition \ref{notC_0} it follows that
\begin{eqnarray*}
(G(t,s)h(\varphi))(x)\ge
C_{d_1,d_2}h((G(t,s)\varphi)(x))\quad\textrm{for each}\,\,d_1\le s\le
t\le d_2,\,x\in\R^d.
\end{eqnarray*}
Note that the function $(G(t,\cdot)h(\varphi))(x)$ is integrable in $[d_1,t]$ for any $t\in (d_1,d_2]$ since it
can be bounded from above by $-(G(t,\cdot)\mathcal{A}(\cdot)\varphi)(x)$.

Let us now fix $x\in \R^d$, $t\in [d_1,d_2]$ and define the function $\beta:[0,r_0)\to\R$,
where $r_0\in\R\cup\{+\infty\}$ satisfies $t-r_0=\inf I$, by setting
\begin{eqnarray*}
\beta(r):=(G(t, t-r)\varphi)(x), \quad r\in [0,r_0).
\end{eqnarray*}
Then, $\beta$ is measurable since it is the limit of the sequence of
the continuous functions $r \mapsto (G(t, t-r)\varphi_n)(x)$ (see Corollary \ref{cont-c_0}).

Fix $b=t-d_1$. From
estimate \eqref{est_1}, the condition
$(\mathcal{A}(s)\varphi)(x)\leq -h(\varphi(x))$, for any $s\in [d_1,d_2]$, and all the above
remarks, we deduce that
\begin{align*}
\beta(b)-\beta(0)&\leq
-\int_{t-b}^{t}(G(t,\sigma)h(\varphi))(x)d\sigma\\
&\leq -C_{d_1,d_2}\int_{t-b}^{t}h((G(t,\sigma)\varphi)(x))d\sigma\\
&= -C_{d_1,d_2}\int_0^b h(\beta(\sigma))d\sigma.
\end{align*}
Let $y(\cdot)=y(\cdot;x)$ denote the solution
of the following Cauchy problem
\begin{equation}
\left\{
\begin{array}{ll}
y'(r)=- C_{d_1,d_2} h(y(r)), \quad r\geq 0,\\[1mm]
y(0)=\varphi(x).
\end{array}
\right.
\label{pb-y}
\end{equation}
Then, (i) $\beta(r)\leq y(r)$ for every $r\in [0,b]$ and (ii) $y(\cdot,x)$ is bounded from
above in $[\delta,+\infty)$ for every $\delta>0$, uniformly with
respect to $x\in \R^d$, that is there exists
$\bar{y}=\bar{y}(\delta)>0$, independent of the initial datum
$\varphi(x)$, such that $y(r,x)\leq \bar{y}$ for every $r\ge\delta$.
To establish these properties it suffices to argue as in \cite[Theorem
3.3]{Lun10Com} and \cite[Theorem
5.1.5]{BerLor07Ana}. For the reader's convenience we provide here some details.
To prove (i) one argues by contradiction and supposes that there exists $s_0\in (0,b)$ such that
$\beta(s_0)>y(s_0)$. Then, there exists an interval $L$ containing $s_0$ where
$\beta>y$. It suffices to observe that the inequality
\begin{eqnarray*}
\beta(s_2)-\beta(s_1)\le -C_{d_1,d_2}\int_{s_1}^{s_2}h(\beta(\sigma))d\sigma,\qquad\;\,s_1,s_2\in [0,d],
\end{eqnarray*}
implies that the function $s\mapsto \beta(s)+C_{d_1,d_2}ms$, where $m:=\left (\min_{\R_+}h\right )$, is decreasing.
Thus,
\begin{align*}
\lim_{s\to s_0^-}(\beta(s)+C_{d_1,d_2}m s)&\ge\beta(s_0)+C_{d_1,d_2}m s_0\\
&>y(s_0)+C_{d_1,d_2}m s_0=
\lim_{s\to s_0^-}(y(s)+C_{d_1,d_2}ms),
\end{align*}
so that $\beta>y$ in a left neighborhood of $s_0$.
If we set $a=\inf L$, then $\beta(a)\le y(a)$.
We get to a contradiction observing that
\begin{eqnarray*}
\beta(s)-\beta(a)\le -C_{d_1,d_2}\int_a^sh(\beta(\sigma))d\sigma,\qquad\;\,
y(s)-y(a)= -C_{d_1,d_2}\int_a^sh(y(\sigma))d\sigma,\qquad\;\,
\end{eqnarray*}
which yields
\begin{eqnarray*}
\beta(s)-y(s)\le C_{d_1,d_2}\int_a^s\left (h(y(\sigma))-h(\beta(\sigma))\right )d\sigma,\qquad\;\,s\in L,
\end{eqnarray*}
which is a contradiction since the left-hand side is positive while the right-hand side is negative.

To prove (ii) we rewrite problem \eqref{pb-y} into the following equivalent form:
\begin{equation}
-\int_{\varphi(x)}^{y(t;x)}\frac{dz}{h(z)}=C_{d_1,d_2}t.
\label{AA}
\end{equation}
Suppose that $\varphi(x)>x_h$ (where, we recall, $x_h$ is the unique positive zero of $h$) and fix $\delta>0$ and $t\ge\delta$. Since $1/h$ is integrable in a neighborhood of $+\infty$, using the above formula we
conclude that
\begin{equation}
\int_{y(t;x)}^{+\infty}\frac{dz}{h(z)}\ge C_{d_1,d_2}t\ge C_{d_1,d_2}\delta.
\label{A}
\end{equation}
Since $1/h$ is not integrable in a right neighborhood of $x_h$, there exists a unique $M>x_h$ such that
\begin{equation}
\int_M^{+\infty}\frac{dz}{h(z)}=C_{d_1,d_2}\delta.
\label{B}
\end{equation}
 From \eqref{A} and \eqref{B} it follows that $y(t;x)\le M$ for any $t\ge\delta$.

Suppose now that $\varphi(x)<x_h$. Then, from \eqref{AA} it follows that $y(t;x)\le x_h$ for any $t\ge\delta$.
The proof of property (ii) is now complete.

The properties (i) and (ii) now imply that $(G(t,t-r)\varphi)(x)\leq \bar{y}$ for every $r\in [\delta,t-d_1]$.
Let $R>0$ and assume that $s\in [d_1,t-\delta]$. Then,
$(G(t,s)\varphi)(x)\le\bar{y}$. Hence,
\begin{align*}
g_{t,s}(x,\R^d\setminus B(0,R))&=\int_{\R^d\setminus B(0,R)}g_{t,s}(x,dy)\\
&\leq \frac{1}{\inf\{\varphi(y): |y|\geq R\}}\int_{\R^d\setminus B(0,R)}\varphi(y)g_{t,s}(x,dy)\\
& \leq \frac{(G(t,s)\varphi)(x)}{\inf\{\varphi(y): |y|\geq R\}}\\
&\leq \frac{\bar{y}}{\inf\{\varphi(y): |y|\geq R\}},
\end{align*}
and $\inf\{\varphi(y): |y|\geq R\}$ tends to $+\infty$ as $R\to+\infty$.
It follows that, for any $\varepsilon>0$, $g_{t,s}(x,\Rd\setminus B(0,R))\le\varepsilon$, for any $x\in\Rd$, if $R$ is sufficiently large
and $s\in [d_1,t-\delta]$. The arbitrariness of $\delta>0$ allows us to conclude through Proposition \ref{tight_equiv_comp}.

Let us now consider the general case when the infimum $c_0$ of $c$ is negative.
We introduce the evolution operator $(P(t,s))$=$(e^{c_0(t-s)}G(t,s))$ which is associated with the elliptic operator
${\mathcal A}_{-c_0}$. Note that $\mathcal{A}_{-c_0}$ satisfies assumption \eqref{cond-comp} with $h$ replaced by
$h-c_0$. Moreover, $\mathcal{A}_{-c_0}(t)W-(c_0+\mu) W\ge 0$ in $[d_1,d_2]\times\R^d\setminus B(0,R)$.
Since Hypothesis \ref{hyp1-bis} is trivially fulfilled, we conclude that the operator $P(t,s)$ is compact
for any $s,t\in [d_1,d_2]$ with $s<t$. As a byproduct $G(t,s)$ is compact in $C_b(\Rd)$ for the same values of $s$ and $t$.
This accomplishes the proof.
\end{proof}

\begin{rmk}
\rm
In the conservative case treated in \cite{Lun10Com}, the existence of the function $W$ as in Proposition \ref{notC_0} is not needed, since $g_{t,s}(x,\R^d)=1$ for every $t>s$ and every $x \in \R^d$. Hence, \eqref{l_bound} is
trivially satisfied.
\end{rmk}

\subsection{A consequence of Theorem \ref{thm-comp}}
\label{subsect-consequences}
Let us prove the following result.

\begin{thm}
Under the assumptions of Theorem $\ref{thm-comp}$, for any $s,t\in I$, with $s<t$, $G(t,s)$ preserves neither
$C_0(\R^d)$ nor $L^p(\R^d)$ $(p\in [1,+\infty))$.
\end{thm}

\begin{proof}
Let $(f_n)$ be a sequence of smooth functions such that $\chi_{B(0,n)}\le f_n\le\chi_{B(0,2n)}$ for any $n\in\N$,
and fix $s,t\in I$ with $s<t$. From formula \eqref{repres-formula} and the dominated convergence theorem, it
follows immediately that $G(t,s)f_n$ converges pointwise in $\R^d$ to $G(t,s)\one$ as $n\to +\infty$. Since $G(t,s)$ is
a compact operator, $G(t,s)f_n$ actually converges uniformly in $\R^d$ to $G(t,s)\one$. Since $G(t,s)$ is
bounded in $C_b(\R^d)$, if it preserved $C_0(\R^d)$ the function $G(t,s)\one$ would tend to $0$ as $|x|\to +\infty$, but this is not
the case. Indeed, formula \eqref{l_bound} shows that $G(t,s)\one$ is bounded from below by a positive constant.

To prove that $G(t,s)$ does not preserve $L^p(\Rd)$, we denote by $K$ any positive constant such that
$g_{t,s}(x,\R^d)\ge K$ for any $x\in\R^d$. By Proposition \ref{tight_equiv_comp}, we can fix $R>0$ such that
$g_{t,s}(x, \R^d\setminus B(0,R))\le K/2$. By difference it follows that
\begin{eqnarray*}
(G(t,s)\chi_{B(0,R)})(x)=g_{t,s}(x,\Rd)-g_{t,s}(x,\R^d\setminus B(0,R))\ge \frac{K}{2},\qquad\;\,x\in\Rd.
\end{eqnarray*}
Hence, $G(t,s)\chi_{B(0,R)}$ does not belong to $L^p(\Rd)$.
\end{proof}

\subsection{An extension of Corollary \ref{cont-c_0} to $C_b(\Rd)$}

An insight in the proof of Theorem \ref{thm-comp} shows that, if $c\geq 0$ and
\begin{equation}\label{M}
(\mathcal{A}(t)\varphi)(x)\leq -h(\varphi(x)),\qquad\;\, t\in J,\;\, |x|\geq R,
\end{equation}
for some interval $J\subset I$ and some $R>0$, then
\begin{equation}\label{M12}
M_{J,\rho,\delta}=\sup_{{(t,s)\in\Lambda_J, t-s>\delta}\atop{|x|\leq \rho}}(G(t,s)\varphi)(x)<+\infty,
\end{equation}
for any $\delta,\rho>0$.

Actually, as in \cite{KunLorLun09Non}, slightly modifying the proof, we can improve \eqref{M12}, removing the
condition $t-s\ge\delta$. For this purpose, in fact, we just need a weaker assumption than \eqref{M}.
More precisely we will assume that the following hypothesis is satisfied.
\begin{hyp}
\label{hyp4bis}
For every bounded interval $J\subset I$ there exist a positive function
$\varphi=\varphi_J \in C^2(\R^d)$ diverging to $+\infty$ as $|x|\to +\infty$ and a positive
constant $M_J$ such that
$$(\mathcal{A}(t)\varphi)(x)\leq M_J, \qquad\;\, (t,x)\in J\times \Rd.$$
\end{hyp}

\begin{prop}\label{slight_hyp}
Let $c\geq 0$ and assume that Hypothesis $\ref{hyp4bis}$ holds.
Then $G(\cdot,\cdot)\varphi$ is bounded in $\Lambda_J\times B(0,\rho)$ for every $\rho>0$.
\end{prop}
\begin{proof}
We can repeat the proof of Theorem \ref{thm-comp} until formula \eqref{est1}, so that we have
\begin{align*}
 \varphi_n(x)-(G(t,s)\varphi_n)(x) \geq -\int_s^t d\sigma\int_{\R^d}\psi_n'(\varphi(y))(\mathcal{A}(\sigma)\varphi)(y)g_{t,\sigma}(x,dy),
\end{align*}
for any $(t,s)\in \Lambda_J$ and any $x\in \Rd$.
Since $(\psi_n'(\varphi))(y)$ is nonnegative using the assumptions we get
\begin{align}
\label{est11}
\varphi_n(x)-(G(t,s)\varphi_n)(x) \geq -M_J\int_s^t d\sigma\int_{\R^d}\psi_n'(\varphi(y))g_{t,\sigma}(x,dy).
\end{align}
Letting $n\to +\infty$ in \eqref{est11} we get
\begin{eqnarray*}
(G(t,s)\varphi)(x)\leq \varphi(x)+M_J(t-s),
\end{eqnarray*}
for any $s,t\in J$, such that $s\leq t$, and any $x\in \R^d$. The claim follows.
\end{proof}

Let us now give the definition of tightness for a one-parameter family of Borel measures. We stress that in the particular case of probability measures
our definition agrees with the classical one.

\begin{defi}
Let ${\mathcal F}=\{\mu_s: s\in F\}$ be a family of finite Borel measures on $\Rd$. We say that ${\mathcal F}$ is tight if,
for any $\varepsilon>0$, there exists $M>0$ such that $\mu_t(\Rd\setminus B(0,M))\le\varepsilon$ for any $t\in F$.
\end{defi}

As a consequence of Proposition \ref{slight_hyp} we obtain the following result.

\begin{prop}\label{tight-loc}
Assume that Hypotheses $\ref{hyp1-bis}$ and $\ref{hyp4bis}$ hold. Then,
for every bounded interval $J\subset I$ and every $R>0$, the family of measures
$\{g_{t,s}(x,dy):\,(t,s,x)\in \Lambda_J\times\overline{B(0,R)}\}$ is tight.
\end{prop}
\begin{proof}
In the case when $c\geq 0$ the proof is similar to that of \cite[Lemma 3.5]{KunLorLun09Non}.\\
If $c_0<0$ we can consider the evolution operator $(P(t,s))=(e^{c_0(t-s)}G(t,s))$ associated with the
elliptic operator $\mathcal{A}_{-c_0}(t)$, whose potential term is
nonpositive and satisfies Hypotheses \ref{hyp1-bis} and \ref{hyp4bis}. Then, the family of measures
$\{p_{t,s}(x,dy):\,(t,s,x)\in \Lambda_J\times \overline{B(0,R)}\}$ associated with $P(t,s)$ satisfies the claim as well as the family
$\{g_{t,s}(x,dy):\,(t,s,x)\in \Lambda_J\times \overline{B(0,R)}\}$ since $p_{t,s}(x,dy)=e^{c_0(t-s)}g_{t,s}(x,dy)$ for every $(t,s,x)\in\Lambda\times\R^d$.
\end{proof}

The following result allows us to extend
the continuity property of the function $G(\cdot,\cdot)f$, stated in
Corollary \ref{cont-c_0} for $f\in C_0(\R^d)$, to the case when $f$ is merely bounded and continuous in $\R^d$.

\begin{prop}\label{prop3.6KLL}
Assume that Hypotheses $\ref{hyp1-bis}$ and $\ref{hyp4bis}$ hold. Let $(f_n)\subset C_b(\Rd)$ be a bounded
sequence converging to $f\in C_b(\Rd)$ locally uniformly in $\R^d$. Then, $G(\cdot,\cdot)f_n$ converges to $G(\cdot,\cdot)f$ locally uniformly in $\Lambda\times\R^d$.
\end{prop}
\begin{proof}
The proof can be obtained as the proof of \cite[Proposition 3.6]{KunLorLun09Non}, taking Proposition \ref{tight-loc} into account.
\end{proof}

\begin{thm}
Under the assumptions of Proposition $\ref{prop3.6KLL}$,
the function $G(\cdot,\cdot)f$ is continuous in $\Lambda\times\R^d$, for every $f\in C_b(\R^d)$.
\end{thm}
\begin{proof}
Let $f\in C_b(\R^d)$, by Proposition \ref{prop3.6KLL} we can find a sequence
of bounded functions $f_n\in C^\infty_c(\R^d)$ converging to $f$ locally uniformly in $\R^d$
such that $G(\cdot,\cdot)f_n$ converges to $G(\cdot,\cdot)f$ locally uniformly.
Since any function $G(\cdot,\cdot)f_n$ is continuous in $\Lambda\times\R^d$ for any $n\in\mathbb N$ by Corollary
\ref{cont-c_0}, the assertion follows at once.
\end{proof}

\section{Invariance of $C_0(\R^d)$ and $L^p(\R^d)$}
\label{sect-5}
In Subsection \ref{subsect-consequences} we have obtained some conditions which imply that neither $C_0(\R^d)$ nor
$L^p(\R^d)$ is preserved by $G(t,s)$. Here, we provide sufficient conditions for $C_0(\R^d)$ and $L^p(\R^d)$ be preserved
by $G(t,s)$.

\subsection{Invariance of $C_0(\Rd)$}
\begin{prop}
Fix $a,b\in I$ such that $a<b$.
Assume that there exist a strictly positive function $V\in C^2(\R^d)$
and $\lambda_0>0$ such that $\lim_{|x|\to+\infty}V(x)=0$ and
$\lambda_0 V(x)-\mathcal{A}(t)V(x)\geq 0$ for every $(t,x)\in [a,b]\times \R^d$. Then,
$G(t,s)$ preserves $C_0(\R^d)$ for any $(t,s)\in\Lambda_{[a,b]}$.
\end{prop}

\begin{proof}
Fix $s\in [a,b]$.
It suffices to prove the statement for $f\in C_c(\R^d)$ since we may approximate an arbitrary
$f\in C_0(\R^d)$ by a sequence $(f_n) \subset C_c(\R^d)$ with respect to the sup-norm in $\Rd$, and $G(t,s)f_n$
converges uniformly to $G(t,s)f$ for every $t\geq s$.
It is not restrictive to suppose $f\geq 0$ otherwise we consider its positive and negative part.
Fix $R>0$, assume that
$\supp f\subset B(0,R)$ and consider the unique bounded classical solution $u$ of the Cauchy problem \eqref{NonA}.
Let $\delta=\inf_{x\in B(0,R)}V(x)>0$ and $z(t,x)=e^{-\lambda_0(t-s)}u(t,x)-\delta^{-1}\Vert f
\Vert_{\infty}V(x)$. Then, the function $z \in C_b([s,b]\times \Rd)\cap C^{1,2}((s,b]\times \Rd)$
satisfies
\begin{displaymath}
\left\{
\begin{array}{ll}
D_tz(t,x)-\mathcal{A}_{\lambda_0}(t)z(t,x)\leq 0,\quad&(t,x)\in (s,b]\times \R^d,\\[1mm]
z(s,x)\leq 0,& x \in \R^d.
\end{array}
\right.
\end{displaymath}
Therefore, applying Proposition \ref{primax} (with $\mathcal{A}$ replaced with $\mathcal{A}_{\lambda_0}$)
we get $z\leq 0$, i.e.,
\begin{eqnarray*}
0\leq u(t,x) \leq e^{\lambda_0(t-s)}\delta^{-1}\Vert f\Vert_\infty V(x),\qquad\;\, s\le t\le b,\;\,x\in\R^d,
\end{eqnarray*}
which implies that $u\in C_0(\R^d)$.
\end{proof}

\subsection{Invariance of $L^p(\Rd)$}
We now study the invariance of $L^p(\R^d)$ under the action of the operator $G(t,s)$. For this purpose, besides Hypothesis \ref{hyp1}, we assume
the following additional assumption on the coefficients of the operator $\mathcal{A}$.

\begin{hyp}
\label{hyp-5.1}
The diffusion coefficients $q_{ij}$ $(i,j=1,\ldots,d)$ are continuously differentiable with respect to the spatial variables in
$[a,b]\times\R^d$ for some $[a,b]\subset I$.
\end{hyp}
\noindent
Let us define
\begin{equation}
\beta_i(t,x)=b_i(t,x)-\sum_{j=1}^dD_jq_{ij}(t,x),\quad (t,x)\in I\times\R^d,\;\,i=1,\ldots,d,
\label{coeff-beta}
\end{equation}
then the following result holds.
\begin{thm}
\label{thm:2}
Fix $a,b\in I$, with $a<b$.
Suppose that the the drift coefficients $b_j$ $(j=1,\ldots,d)$ are continuously differentiable in
$[a,b]\times\Rd$ and the  derivative $D_{ij}q_{ij}$ $(i,j=1,\ldots,d)$ exists in $[a,b]\times\Rd$.
Further, assume that there exists $K>0$ such that
\begin{equation}
c(t,x)+{\rm div}_x\beta(t,x)\geq -K,\qquad\;\,(t,x)\in [a,b]\times \R^d.
\label{cond-diverg}
\end{equation}
Then, for every $1\le p<+\infty$, $L^p(\R^d)$ is invariant under $G(t,s)$ for any $(t,s)\in\Lambda_{[a,b]}$.
Moreover,
\begin{equation}
\|G(t,s)f\|_{L^p(\Rd)}\le e^{K_p(t-s)}\|f\|_{L^p(\Rd)},\quad\;\,a\le s\le t\le b,
\label{stima-norma-op-Lp}
\end{equation}
where $pK_p=K-(p-1)c_0$.
\end{thm}

\begin{proof}
Fix $s\in [a,b]$. We prove the assertion
for nonnegative $f\in C_c^\infty(\R^d)$. The density of
$C_c^\infty(\R^d)$ in $L^p(\R^d)$ ($p\in [1,+\infty)$) combined with
the estimate $|G(t,s)f|\le G(t,s)|f|$ (see \eqref{repres-formula}) then allows us
to extend the result to any $f\in L^p(\Rd)$.

Let $u(t,x)=(G(t,s)f)(x)$ and, for any $n\in\N$, let $u_n=G_n^D(\cdot,s)f$ be the classical solution of the Cauchy-Dirichlet problem
\eqref{pb-approx-Dirichlet}. By Theorem \ref{thm-1.3}(ii), $u_n$ is nonnegative in $[s,+\infty)\times\R^d$ and therein converges to $u$ pointwise as $n\to +\infty$.

Let us prove that
\begin{equation}
\Vert u_n(t,\cdot)\Vert_{L^p(B(0,n))}\leq e^{K_p(t-s)}\Vert u_n(s,\cdot)\Vert_{L^p(B(0,n))}
=e^{K_p(t-s)}\Vert f\Vert_{L^p(B(0,n))},
\label{crucial}
\end{equation}
for any $t\in [s,b]$ and any $n\in\mathbb N$ such that ${\rm supp}(f)\subset B(0,n)$.
We first assume that $p\neq 1$ and set $u_n^\varepsilon=u_n+\varepsilon$. Then
\begin{align*}
\frac{d}{dt}\Vert u_n^\varepsilon(t,\cdot)\Vert_{L^p(B(0,n))}^p&=
p\int_{B(0,n)}(u_n^\varepsilon(t,\cdot))^{p-1}\mathcal{A}(t)u_n(t,\cdot)\,dx.
\end{align*}
Integrating by parts and using Hypotheses \ref{hyp1} and \ref{hyp-5.1}, we get
\begin{align}
&\int_{B(0,n)}(u_n^\varepsilon (t,\cdot))^{p-1}\mathcal{A}(t)u_n (t,\cdot)\,dx\notag\\ 
=&\varepsilon^{p-1}\int_{\partial B(0,n)}\langle Q(t,\cdot)\nabla_x u_n(t,\cdot),\nu\rangle\,dx\notag\\
&-(p-1)\int_{B(0,n)}\big( u_n^\varepsilon(t,\cdot)\big)^{p-2}\langle Q(t,\cdot)\nabla_xu_n(t,\cdot),\nabla_xu_n(t,\cdot)\rangle dx\notag\\
&+\frac{1}{p}\int_{B(0,n)}\langle
\beta(t,\cdot),\nabla_x(u^\varepsilon_n(t,\cdot))^p\rangle\,dx
-\int_{B(0,n)}c(t,\cdot)(u_n^\varepsilon (t,\cdot))^p\,dx\notag\\
&+\varepsilon\int_{B(0,n)}c(t,\cdot)(u_n^\varepsilon (t,\cdot))^{p-1}\,dx\notag\\
 \leq &\varepsilon^{p-1}\int_{\partial B(0,n)}\langle Q(t,\cdot)\nabla_xu_n(t,\cdot),\nu\rangle\,dx
 +\frac{\varepsilon^p}{p}\int_{\partial B(0,n)}\langle \beta(t,\cdot),\nu\rangle\,dx\notag\\
&- \frac{1}{p}\int_{B(0,n)}(u^\varepsilon_n(t,\cdot))^p \left (p\, c(t,\cdot)+\textrm{div}_x\beta(t,\cdot)\right )\,dx\notag\\
&+\varepsilon\int_{B(0,n)}c(t,\cdot)(u_n^\varepsilon (t,\cdot))^{p-1}\,dx.
\label{int-luglio-0}
\end{align}
We now observe that
\begin{align*}
p\, c(t,\cdot)+\textrm{div}_x\beta(t,\cdot)&=p\left (c(t,\cdot)-c_0\right )+pc_0+\textrm{div}_x\beta(t,\cdot)\\
&\ge \left (c(t,\cdot)-c_0\right )+\textrm{div}_x\beta(t,\cdot)+pc_0\ge -pK_p.
\end{align*}
Hence, from \eqref{int-luglio-0} we get
\begin{align}
&\int_{B(0,n)}(u_n^\varepsilon (t,\cdot))^{p-1}\mathcal{A}(t)u_n^\varepsilon (t,\cdot)\,dx\notag\\  \leq &\varepsilon^{p-1}\int_{\partial B(0,n)}\langle
Q(t,\cdot)\nabla_x u_n(t,\cdot),\nu\rangle\,dx+\frac{\varepsilon^p}{p}\int_{\partial B(0,n)}\langle
\beta(t,\cdot),\nu\rangle\,dx\notag\\
&+K_p\int_{B(0,n)}(u^\varepsilon_n(t,\cdot))^p\,dx
+\varepsilon p\int_{B(0,n)}c(t,\cdot)(u_n^\varepsilon (t,\cdot))^{p-1}\,dx,
\label{int-luglio}
\end{align}
for any $t\in [s,b]$,
where $\nu=\nu(x)$ is the outward unit normal at $x \in \partial B(0,n)$.
If we set
\begin{align*}
g_n^{\varepsilon,p}(t):=&p\,\varepsilon^{p-1}\int_{\partial B(0,n)}\langle Q(t,\cdot)\nabla_xu_n(t,\cdot),\nu\rangle\,dx
 +\varepsilon^p\int_{\partial B(0,n)}\langle \beta(t,\cdot),\nu\rangle\,dx\\
 &+\varepsilon p\int_{B(0,n)}c(t,\cdot)(u_n^\varepsilon (t,\cdot))^{p-1}\,dx,
\end{align*}
from \eqref{int-luglio} we get
\begin{eqnarray*}
\frac{d}{dt}\Vert u_n^\varepsilon(t,\cdot)\Vert_{L^p(B(0,n))}^p \leq g_n^{\varepsilon,p}(t)
+pK_p\Vert u_n^\varepsilon(t,\cdot)\Vert_{L^p(B(0,n))}^p,\quad t\in [s,b].
\end{eqnarray*}
Hence, we easily deduce that
\begin{eqnarray*}
\Vert u_n^\varepsilon(t,\cdot)\Vert_{L^p(B(0,n))}^p\leq e^{pK_p(t-s)}\Vert u_n^\varepsilon(s,\cdot)\Vert_{L^p(B(0,n))}^p+\int_s^t e^{pK_p(t-\tau)}g_n^{\varepsilon,p}(\tau)\,d\tau,
\end{eqnarray*}
and, by dominated convergence, \eqref{crucial} follows at once.

To prove \eqref{crucial} for $p=1$ it suffices to write it for $p>1$ and, then, let $p\to 1^+$ since
\begin{eqnarray*}
\lim_{p\to 1^+}\|\psi\|_{L^p(B(0,n))}=\|\psi\|_{L^1(B(0,n))},
\end{eqnarray*}
for any $\psi\in C(\overline{B(0,n)})$.

Now, let $v_n(t,x)=u_n(t,x)\chi_{B(0,n)}$. Then, $\lim_{n\to+\infty}v_n(t,x)=u(t,x)$ for $(t,x)\in (s,+\infty)\times \R^d$ and
\begin{align*}
\Vert u(t,\cdot)\Vert_{L^p(\R^d)}^p&\leq \liminf_{n \to +\infty}\Vert v_n(t,\cdot)\Vert_{L^p(\R^d)}^p = \liminf_{n \to +\infty}\Vert u_n(t,\cdot)\Vert_{L^p(B(0,n))}^p\\
& \leq \liminf_{n \to +\infty} e^{pK_p(t-s)}\Vert f\Vert_{L^p(B(0,n))}^p = e^{pK_p(t-s)}\Vert f\Vert_{L^p(\R^d)}^p,
\end{align*}
for any $t\in [s,b]$. Therefore,
$G(t,s)\in \mathcal{L}(L^p(\R^d))$
for $1\le p<+\infty$ and $t\in [s,T]$, and it satisfies \eqref{stima-norma-op-Lp}.
This completes the proof.
\end{proof}

The condition assumed in Theorem \ref{thm:2} is a sort of compensation between the diffusion coefficients, the drift, the potential of the operator $\mathcal{A}$. Note that in the case when $c\equiv 0$ and $q_{ij}$ ($i,j=1,\ldots,d$) are constant with respect to the spatial variables, such a condition reduces to the request that
the spatial divergence of the drift $b$ is bounded from below.
Slightly modifying the proof of the previous theorem, we can give another sufficient condition for $L^p(\R^d)$ be preserved
by the action of the evolution operator $G(t,s)$, which applies to some situation where condition
\eqref{cond-diverg} is not satisfied (see Remark \ref{rem:6.7}).

\begin{thm}
\label{thm-variante}
Fix $p>1$, $a,b\in I$ with $a<b$. Assume that
\begin{equation}
\frac{|\beta(t,x)|^2}{4(p-1)\eta(t,x)}-c(t,x)\le K_p',\quad\;\,(t,x)\in [a,b]\times\R^d,
\label{assumption}
\end{equation}
$($see \eqref{coeff-beta}$)$ for some positive constant $K_p'$. Then, $L^p(\R^d)$ is invariant under $G(t,s)$ for any $(t,s)\in\Lambda_{[a,b]}$.
Moreover,
\begin{eqnarray*}
\|G(t,s)f\|_{L^p(\Rd)}\le e^{K_p'(t-s)}\|f\|_{L^p(\Rd)},\quad\;\,a\le s\le t\le b.
\end{eqnarray*}
\end{thm}

\begin{proof}
The main difference with respect to the proof of Theorem \ref{thm:2} is in the estimate of
the term
\begin{eqnarray*}
I:=\int_{B(0,n)}(u_n^{\varepsilon}(t,x))^{p-1}\langle \beta(t,x),\nabla_xu_n(t,x)\rangle\,dx.
\end{eqnarray*}
Using H\"older and Young's inequality we can estimate
\begin{align*}
I\leq & \int_{B(0,n)}\sqrt{\eta(t,\cdot)}(u_n^{\varepsilon}(t,\cdot))^{\frac{p}{2}-1}|\nabla_xu_n(t,\cdot)|\frac{1}{\sqrt{\eta(t,\cdot)}}|\beta(t,\cdot)|
(u_n^{\varepsilon}(t,\cdot))^{\frac{p}{2}}\,dx\\
\le & \left (\int_{B(0,n)}\eta(t,\cdot)(u_n^{\varepsilon}(t,\cdot))^{p-2}|\nabla_xu_n(t,\cdot)|^2\,dx\right )^{\frac{1}{2}}\hskip -2pt
\left (\int_{B(0,n)}\frac{|\beta(t,\cdot)|^2}{\eta(t,\cdot)}(u_n^{\varepsilon}(t,\cdot))^pdx\right )^{\frac{1}{2}}\\
\le &\delta\int_{B(0,n)}\eta(t,\cdot)(u_n^{\varepsilon}(t,\cdot))^{p-2}|\nabla_xu_n(t,\cdot)|^2\,dx
+\frac{1}{4\delta}\int_{B(0,n)}\frac{|\beta(t,\cdot)|^2}{\eta(t,\cdot)}(u_n^{\varepsilon}(t,\cdot))^pdx,
\end{align*}
for any $\delta>0$. Hence,
\begin{align*}
&\int_{B(0,n)}(u_n^\varepsilon (t,\cdot))^{p-1}\mathcal{A}(t)u_n(t,\cdot)\,dx\\
\le  & \varepsilon^{p-1}\int_{\partial B(0,n)}\langle Q(t,\cdot)\nabla_xu_n(t,\cdot),\nu\rangle\,dx
 +\int_{B(0,n)}\left (\frac{|\beta(t,\cdot)|^2}{4\delta\eta(t,\cdot)}-c(t,\cdot)\right )(u^\varepsilon_n(t,\cdot))^p\,dx\\
&-\left (p-1-\delta\right )\int_{B(0,n)}\eta(t,\cdot)(u_n^\varepsilon(t,\cdot))^{p-2}|\nabla_xu_n(t,\cdot)|^2 dx\\
&+\varepsilon\int_{B(0,n)}c(t,\cdot)(u_n^\varepsilon (t,\cdot))^{p-1}\,dx.
\end{align*}
The optimal choice $\delta=p-1$ gives
\begin{align*}
&\int_{B(0,n)}(u_n^\varepsilon (t,\cdot))^{p-1}\mathcal{A}(t)u_n(t,\cdot)\,dx\\
\le  & \varepsilon^{p-1}\int_{\partial B(0,n)}\langle Q(t,\cdot)\nabla_xu_n(t,\cdot),\nu\rangle\,dx
+K_p'\int_{B(0,n)}(u^\varepsilon_n(t,\cdot))^p\,dx\\
&+\varepsilon\int_{B(0,n)}c(t,\cdot)(u_n^\varepsilon (t,\cdot))^{p-1}\,dx.
\end{align*}

Now, we can conclude arguing as in the proof of Theorem \ref{thm:2}.
\end{proof}

\section{Examples}
\label{sect-6}
In this section we exhibit some classes of operators to which the main results of this paper apply.

\subsection{A class of operators to which Theorem \ref{thm-comp} applies}

Let $\mathcal{A}$ be the differential operator defined by
\begin{equation}
(\mathcal{A}(t)\psi)(x)=\omega(t)(1+|x|^2)^k\Delta\psi(x)+\langle b(t,x),\nabla
\psi(x)\rangle-c(t,x)(1+|x|^2)^m\psi(x),
\label{op-ex-1}
\end{equation}
for any $(t,x)\in I\times\R^{d}$,
on smooth functions $\psi:\R^d\to\R$.
\begin{hyp}
\begin{enumerate}[\rm (i)]
\item
$k,m\in\N$;
\item
$\omega \in C^{\alpha/2}_{\rm loc}(I)$ satisfies $\inf_{t\in I}\omega(t)>0$, $b\in C^{\alpha/2,\alpha}_{\rm loc}(I\times\R^{d},\R^d)$ and
$c\in C^{\alpha/2,\alpha}_{\rm loc}(I\times\R^{d})$ is positive and bounded;
\item
there exist $l\in\N$ such that $l>(m+2)\vee k$, $R>0$ and a continuous function $C_1: I\to
(0,+\infty)$ such that
\begin{eqnarray*}
\langle b(t,x),x\rangle \le
-C_1(t)(1+|x|^2)^l,\qquad\;\,(t,x)\in I\times\R^{d}\setminus B(0,R).
\end{eqnarray*}
\end{enumerate}
\end{hyp}
Under such assumptions the operator $G(t,s)$ associated with the operator ${\mathcal A}$ in \eqref{op-ex-1} is compact in
$C_b(\R^d)$ for any $t,s\in I$ with $s<t$. To check the claim it suffices to show
that, for any bounded interval $J\subset I$, there exist a positive and bounded smooth function $W:\R^d\to\R$, with positive infimum, a
positive smooth function $\varphi:\R^d\to\R$, blowing up as $|x|\to
+\infty$, an increasing strictly convex function $h:[0,+\infty)\to
\R$, with $1/h$ being integrable in a neighborhood of $+\infty$, and
$\mu\in\R$ such that
\begin{align}
&(i)~(\mathcal{A}(t)W)(x)-\mu W(x)\ge
0,\quad\;\,(t,x)\in J\times\R^d\setminus B(0,R),\notag\\
&(ii)~(\mathcal{A}(t)\varphi)(x)\le
-h(\varphi(x)),\quad\;\,(t,x)\in J\times\Rd. \label{2cond}
\end{align}

We have also to show that Hypothesis \ref{hyp1-bis} is fulfilled. For notational convenience
we set $\omega_0=\sup_{t\in J}\omega(t)$.

To prove the first condition in \eqref{2cond}, we set
$W(x)=1+\frac{1}{1+|x|^2}$ for any $x\in\R^d$. Then,
\begin{align*}
(\mathcal{A}(t)W)(x)-\mu W(x) =&-2d\omega(t)(1+|x|^2)^{k-2}
+8 \omega(t)|x|^2(1+|x|^2)^{k-3}\\
& -2\frac{\langle
b(t,x),x\rangle}{(1+|x|^2)^2}
-\left (c(t,x)(1+|x|^2)^m+\mu\right )\left
(1+\frac{1}{1+|x|^2}\right )\\
\ge &2(1+|x|^2)^{l-2}\bigg\{C_1(t) -d\omega(t)(1+|x|^2)^{k-l}\\
&\qquad\qquad\qquad-c(t,x)(1+|x|^2)^{m-l+2}
-|\mu| (1+|x|^2)^{2-l}\bigg\}\\
\ge &2(1+|x|^2)^{l-2}\bigg\{C_1(t) -d\omega_0R^{2k-2l}-|\mu| R^{4-2l}\\
&\qquad\qquad\qquad\;\;\;\;\;
-R^{2m-2l+4}\sup_{(t,x)\in J\times\Rd}c(t,x)\bigg\},
\end{align*}
for any $(t,x)\in J\times\Rd$.
Hence, condition \eqref{2cond}(i) follows for any $\mu\in\R$, provided we take $R$ sufficiently
large.

Let us now check condition \eqref{2cond}(ii). For this purpose, we
set $\varphi(x)=1+|x|^2$ for any $x\in\R^d$. Then,
\begin{align*}
(\mathcal{A}(t)\varphi)(x)=&2d\,\omega(t)(1+|x|^2)^k+2\langle b(t,x),x\rangle -c(t,x)(1+|x|^2)^{m+1}\\
\le & 2d\,\omega(t)(1+|x|^2)^k-2C_1(t)(1+|x|^2)^l\\
=& (1+|x|^2)^l\left
\{-2C_1(t)+2d\,\omega(t)(1+|x|^2)^{k-l}\right\}\\
\le &2(1+|x|^2)^l\left \{-C_1(t)+\omega_0d(1+|x|^2)^{-1}\right\},
\end{align*}
where in the last inequality we have used the fact that $l\ge k+1$.
We now observe that, for any $\varepsilon>0$, any $a>0$ and any
$p\in\N$, $p\ge 2$, it holds that
\begin{eqnarray*}
ar\le \varepsilon+\varepsilon^{1-p}p^{-p}(p-1)^{p-1}a^p r^p:=
\varepsilon +C_{\varepsilon}a^pr^p,\qquad\;\,r>0.
\end{eqnarray*}
Applying this inequality with
\begin{align*}
r=\frac{1}{1+|x|^2},\qquad\;\,a=\omega_0d,\qquad p=l> 2,
\end{align*}
we can estimate
\begin{align*}
(\mathcal{A}(t)\varphi)(x)\le
&-2(C_1(t)-\varepsilon)(1+|x|^2)^l+
2C_{\varepsilon}(\omega_0d)^l,\qquad (t,x)\in J\times\Rd.
\end{align*}
Fix $2\varepsilon<\inf_{t\in J}C_1(t):=\gamma$. With this
choice of $\varepsilon$ we get
\begin{align*}
(\mathcal{A}(t)\varphi)(x)\le -\gamma(1+|x|^2)^l
+2C_{\varepsilon}(\omega_0d)^l:=-\gamma(1+|x|^2)^l+C_{\varepsilon}'.
\end{align*}
Now, we introduce the function $h:[0,+\infty)\to\R$ defined by
$h(t)=\gamma t^l-C_{\varepsilon}'$ for any $t\ge 0$. Clearly, $h$ is strictly increasing,
convex, $1/h$ is integrable in a neighborhood of $+\infty$.
Moreover, $\mathcal{A}(t)\varphi(x)\le
-h(\varphi(x))$ for any $t\in J$ and any $x\in\R^d$,
i.e., condition \eqref{2cond}(ii) holds true.

Note that, in fact, we have shown that
\begin{eqnarray*}
(\mathcal{A}_{-c}(t)\varphi)(x) \le  -\gamma(1+|x|^2)^l
+C_{\varepsilon}',\qquad\;\,t\in J,\;\,x\in\R^d.
\end{eqnarray*}
In particular, this implies that $\mathcal{A}_{-c}(t)\varphi(x) \le
C_{\varepsilon}'\varphi(x)$ for any $(t,x)\in J\times\Rd$,
which clearly implies Hypothesis \ref{hyp1-bis}.

\subsection{A class of operators to which the results of Section \ref{sect-5} apply}

Let ${\mathcal A}$ be defined by
\begin{equation}
({\mathcal A}(t)\varphi)(x)=(1+|x|^2)^m{\rm Tr}(Q(t,x)D^2\varphi(x))+(1+|x|^2)^rb(t)\langle x,\nabla\varphi(x)\rangle-c(t,x)\varphi(x),
\label{oper-A}
\end{equation}
where $m,r$ are nonnegative constants.
We assume the following set of assumptions on the coefficients of the operator ${\mathcal A}$, on $m$ and $r$.
\begin{hyp}\label{hyp2}
\begin{enumerate}[\rm (i)]
\item
$b\in C^{\alpha/2}_{\rm loc}(I)$, $b(t)\le 0$ for any $t\in I$;
\item
$c\in C^{\alpha/2,\alpha}_{\rm loc}(I\times\Rd)$ and, for any bounded interval $J\subset I$, there exist $C_J\ge 0$ and $q=q_J\ge 0$ such that $c(t,x)\ge
C_J(1+|x|^2)^q$ for any $(t,x)\in J\times\Rd$;
\item
there exists a positive constant $\eta_0$ such that
\begin{eqnarray*}
\langle Q(t,x)\xi,\xi\rangle\geq \eta_0|\xi|^2,\quad \xi\in \Rd,\quad (t,x)\in I\times \Rd.
\end{eqnarray*}
Moreover,
\begin{eqnarray*}
M^{(1)}_J:=\sup_{(t,x)\in J\times\Rd}|Q(t,x)|_{\R^{d^2}}<+\infty,
\end{eqnarray*}
for any bounded interval $J\subset I$;
\item
for any bounded interval $J\subset I$ one of the following conditions is satisfied:
\begin{enumerate}[\rm (a)]
\item
$r> m-1$ and $b(t)<0$ in $J$;
\item
$q_J> m-1$ and $C_J>0$;
\end{enumerate}
\item
there exists a compact set $[a,b]\subset I$  such that $C_{[a,b]}>0$ and $q_{[a,b]}>\max\{r,m-1,1\}$.
\end{enumerate}
\end{hyp}

Under the previous conditions, Hypothesis \ref{hyp1} is satisfied. Of course, we have to check only Hypothesis \ref{hyp1}(iv).
For this purpose we take $\varphi(x)=1+|x|^2$ for any $x\in\Rd$.
As it is easily seen
\begin{align*}
({\mathcal A}(t)\varphi)(x)=&2{\rm Tr}(Q(t,x))(1+|x|^2)^m+2b(t)|x|^2(1+|x|^2)^r-c(t,x)(1+|x|^2),
\end{align*}
for any $(t,x)\in I\times\Rd$.
Hence,
\begin{align*}
({\mathcal A}(t)\varphi)(x)\le
2\sqrt{d}M_K^{(1)}(1+|x|^2)^m+2b(t)|x|^2(1+|x|^2)^r-C_K(1+|x|^2)^{q+1},
\end{align*}
for any $(t,x)\in J\times\Rd$ and any bounded interval $J\subset I$.
Is is now easy to show that, under Hypothesis \ref{hyp2}(iv-a) or \ref{hyp2}(iv-b)
\begin{eqnarray*}
R_J:=\sup_{(t,x)\in J\times\R^d}\mathcal A(t)\varphi(x)< +\infty.
\end{eqnarray*}
Hence, Hypothesis \ref{hyp1}(iv) is satisfied with $\lambda=R_J \vee 0$.

We now consider the function $V:\Rd\to\R$ defined by $V(x)=(1+|x|^2)^{-1}$ for any $x\in\Rd$. A straightforward computation shows that
\begin{align*}
({\mathcal A}(t)V)(x) =& 8\langle Q(t,x)x,x\rangle (1+|x|^2)^{m-3}
-2{\rm Tr}(Q(t,x))(1+|x|^2)^{m-2}\\
&-2b(t)|x|^2(1+|x|^2)^{r-2}-c(t,x)(1+|x|^2)^{-1},
\end{align*}
for any $(t,x)\in I\times\Rd$.
Hence,
\begin{align*}
({\mathcal A}(t)V)(x)\le &8M_{[a,b]}^{(1)}|x|^2(1+|x|^2)^{m-3}
+2\|b\|_{L^{\infty}((a,b))}|x|^2(1+|x|^2)^{r-2}\\
&-C_{[a,b]}(1+|x|^2)^{q-1},
\end{align*}
for any $(t,x)\in [a,b]\times\Rd$. Therefore, taking Hypothesis \ref{hyp2}(v) into account, we can conclude that
\begin{eqnarray*}
\lim_{|x|\to +\infty}\sup_{t\in [a,b]}({\mathcal A}(t)V)(x)=-\infty.
\end{eqnarray*}
In particular, there exists $R>0$ such that $({\mathcal A}(t)V)(x)\le 0$ for any $(t,x)\in [a,b]\times \R^d\setminus\overline{B(0,R)}$.
Therefore, the condition
\begin{eqnarray*}
{\mathcal A}(t)V(x)\le \lambda_0V(x),\qquad\;\,(t,x)\in [a,b]\times\Rd,
\end{eqnarray*}
is satisfied with $\lambda_0=(1+R^2)\left (\sup_{(t,x)\in [a,b]\times B(0,R)}{\mathcal A}(t)V(x)\right )^+$. Here,
$(\,\cdot\,)^+$ denotes the positive part of the quantity in brackets.
As a byproduct, we get the following:
\begin{prop}
Under Hypothesis $\ref{hyp2}$ the evolution
operator $(G(t,s))$ associated with the operator ${\mathcal A}$ in
\eqref{oper-A} preserves $C_0(\R^d)$ for any $s,t\in\Lambda_{[a,b]}$.
\end{prop}

Let us now compute the divergence of the vector field $\beta$ defined in \eqref{coeff-beta} for the operator $\mathcal{A}$ in \eqref{oper-A}.
For this purpose, we assume the following additional assumptions on the coefficients of the operator
$\mathcal{A}$, on $m$, $q$ and $r$.
\begin{hyp}
\label{hyp5}
\begin{enumerate}[\rm (i)]
\item
The diffusion coefficients $q_{ij}$ $(i,j=1,\ldots,d)$ are continuously differentiable in $I\times\R^d$ with respect to
the spatial variables and $\nabla_xq_{ij}$ is bounded in $[a,b]\times\Rd$ $($where $[a,b]$ is as in Hypothesis $\ref{hyp2}(v))$ and any $i,j=1,\ldots,d$. Moreover,
the second-order weak spatial derivatives
$D_{ij}q_{ij}$ $(i,j=1,\ldots,d)$ exist and are bounded functions in $[a,b]\times\Rd$;
\item
$q_{[a,b]}>\max\{r,m,1\}$.
\end{enumerate}
\end{hyp}
Under such additional assumptions we get
\begin{align*}
&{\rm div}_x\beta(t,x)+c(t,x)\\
=&-4m(1+|x|^2)^{m-1}\sum_{i,j=1}^dD_iq_{ij}(t,x)x_j
-(1+|x|^2)^m\sum_{i,j=1}^dD_{ij}q_{ij}(t,x)\\
&-4m(m-1)(1+|x|^2)^{m-2}\langle Q(t,x)x,x\rangle
-2m(1+|x|^2)^{m-1}{\rm Tr}(Q(t,x))\\
&+b(t)(1+|x|^2)^{r-1}\left (d+(2r+d)|x|^2\right )+c(t,x),
\end{align*}
for any $(t,x)\in I\times\Rd$.
Hence, we can estimate
\begin{align*}
&{\rm div}_x\beta(t,x)+c(t,x)\\
\ge &-4mM^{(2)}_{[a,b]}|x|(1+|x|^2)^{m-1}
-(1+|x|^2)^{m}M^{(3)}_{[a,b]}\\
&-4m(m-1)M_{[a,b]}^{(1)}(1+|x|^2)^{m-1}
-2m\sqrt{d}M_{[a,b]}^{(1)}(1+|x|^2)^{m-1}\\
&-\|b\|_{L^{\infty}((a,b))}(1+|x|^2)^{r-1}\left (d+(2r+d)|x|^2\right )
+C_{[a,b]}(1+|x|^2)^q,
\end{align*}
for any $(t,x)\in [a,b]\times\Rd$, where
\begin{align*}
&M_{[a,b]}^{(2)}=\sup_{(t,x)\in [a,b]\times\Rd}\left (\sum_{j=1}^d\left (\sum_{i=1}^d|D_iq_{ij}(t,x)|\right )^2\right )^{\frac{1}{2}},\\
&M_{[a,b]}^{(3)}=\sum_{i,j=1}^d\sup_{(t,x)\in [a,b]\times\R^d}|D_{ij}q_{ij}(t,x)|.
\end{align*}

Due to the conditions imposed on $m,q,r$,
${\rm div}_x\beta(t,x)+c(t,x)$ tends to $+\infty$ as $|x|\to +\infty$ for any $t\in [a,b]$. We have so proved the following.

\begin{prop}
\label{prop-1}
Under Hypotheses $\ref{hyp2}$ and $\ref{hyp5}$ the
evolution operator $G(t,s)$ associated with the operator ${\mathcal A}$ in \eqref{oper-A} preserves $L^p(\Rd)$ for any $p\in [1,+\infty)$ and any $s,t\in \Lambda_{[a,b]}$.
\end{prop}

Finally, observe that, arguing as above, one can easily verify that, if
\begin{equation}
q_{[a,b]}>\max\{m,2r+1-m\},
\label{cond}
\end{equation}
then the condition \eqref{assumption} is fulfilled. Hence,
\begin{prop}
\label{prop-2}
Let Hypotheses $\ref{hyp2}(i)$-$(iv)$ and condition \eqref{cond} be fulfilled.
Further, assume that the diffusion coefficients $q_{ij}$ $(i,j=1,\ldots,d)$ are
continuously differentiable with respect to the spatial variables and
assume that $C_{[a,b]}>0$.
Then, the evolution operator $G(t,s)$ associated with the operator ${\mathcal A}$ in
\eqref{oper-A} preserves $L^p(\Rd)$ for any $s,t\in \Lambda_{[a,b]}$.
\end{prop}

\begin{rmk}
\label{rem:6.7}
\rm
In this example, condition \eqref{cond} trivially implies Hypothesis \ref{hyp5}(ii).
Hence, the difference between Propositions \ref{prop-1} and \ref{prop-2} is just in the smoothness of
the coefficients. In general, as claimed before Theorem
\ref{thm-variante}, even for smooth coefficients, condition \eqref{assumption} may hold also in some situations where
condition \eqref{cond-diverg} is not satisfied.
Consider for instance the operator ${\mathcal A}$ defined by
\begin{eqnarray*}
(\mathcal{A}(t)\varphi)(x)=\Delta\varphi(x)-\frac{t^2+2}{t^2+1}\left (2+\sin(|x|^4)\right )\langle x,\nabla\varphi(x)\rangle
-(t^2+1)(1+|x|^2)^q\varphi(x),
\end{eqnarray*}
for any $(t,x)\in\R^{1+d}$, on smooth functions $\varphi$.

A straightforward computation shows that operator $\mathcal{A}$ satisfies Hypothesis \ref{hyp1}. On the other hand,
condition \eqref{cond-diverg} is satisfied, by any $[a,b]\subset I$ provided that $q_{[a,b]}>2$,
whereas condition \eqref{cond} is satisfied (by any $p>1$ and any $[a,b]$ as above) provided
that $q_{[a,b]}>1$.
\end{rmk}

\subsection*{Acknowledgments} The authors wish to thank A. Lunardi, G. Metafune and D. Pallara for
fruitful discussions.


\begin{thebibliography}{99}

\bibitem{Acq88Evo}
P. Acquistapace,
\newblock {\em Evolution operators and strong solutions of abstract linear parabolic equations}\,,
\newblock{Diff. Int. Eqns.} \textbf{1} (1988), 433-457.

\bibitem{luciana-giorgio-chiara}
L. Angiuli, G. Metafune, C. Spina,
\newblock{\em Feller semigroups and invariant measures}\,,
\newblock{Riv. Mat. Univ. Parma} (to appear).

\bibitem{azencott}
R. Azencott,
\newblock{\em Behaviour of diffusion semigroups at infinity}\,,
\newblock{Bull. Soc. Math. France} \textbf{102} (1974), 193-240.

\bibitem{BerLor07Ana}
M. Bertoldi, L. Lorenzi,
\newblock {Analytical Methods for Markov Semigroups}\,,
\newblock {Chapman \& Hall/CRC, Boca Raton}, 2007.

\bibitem{DaPLun07}
G. Da Prato, A. Lunardi,
\newblock{\em Ornstein-Uhlenbeck operators with time periodic coefficients}\,,
\newblock{J. Evol. Equ.} \textbf{7} (2007), 587-614.

\bibitem{DPR1}
G. Da Prato, M. R\"ockner,
\newblock{\em A note on evolution systems of
measures for time-dependent stochastic differential equations}\,,
\newblock{In: Seminar on Stochastic Analysis, Random Fields and Applications V,}
\newblock{Progr. Probab.} \textbf{59} (2008), Birkh\"auser, Basel, 115-122.

\bibitem{dynkin}
E.B. Dynkin,
\newblock{\em Three classes of infinite-dimensional
diffusions}\,,
\newblock{J. Funct. Anal.} \textbf{86} (1989), 75-110.


\bibitem{Fri64Par}
A. Friedman
\newblock {Partial differential equations of parabolic type}\,,
\newblock {Prentice Hall, Englewood Cliffs, N. J.}, 1964.


\bibitem{GeisLun08}
M. Geissert,  A. Lunardi,
\newblock{\em Invariant measures and maximal $L^2$ regularity for nonautonomous Ornstein-Uhlenbeck equations}\,,
\newblock{J. Lond. Math. Soc. (2)}  \textbf{77}  (2008), 719-740.

\bibitem{GeisLun09}
M. Geissert,  A. Lunardi,
\newblock{\em Asymptotic behavior and hypercontractivity in nonautonomoous
Ornstein-Uhlenbeck equations}\,, \newblock{J. Lond.\ Math. Soc. (2)}
\textbf{79} (2009),  85-106.

\bibitem{ito}
S. It{\^o},
\newblock{\em Fundamental solutions of parabolic differential equations and boundary value problems}\,,
\newblock{Jap. J. Math.} \textbf{27} (1957), 55-102.


\bibitem{KunLorLun09Non}
M. Kunze, L. Lorenzi, A. Lunardi,
\newblock {\em Nonautonomous Kolmogorov parabolic equations with unbounded coefficients}\,,
\newblock {Trans. Amer. Math. Soc.}, \textbf{362} (2010), 169-198.




\bibitem{LadSolUra68Lin}
O.A. Lady$\check{\textrm{z}}$henskaja, V.A. Solonnikov, N.N. Ural'ceva,
\newblock{Linear and quasilinear equations of parabolic type},\, Nauka, Moscow, 1967
\newblock{English transl.: American Mathematical Society},\, Providence, R.I. 1968.

		
\bibitem{LorLunZam10}
L. Lorenzi, A. Lunardi, A. Zamboni,
\newblock{\em Asymptotic behavior in
time periodic parabolic problems with unbounded coefficients}\,,
\newblock{arXiv: 0908.1170v1.}


\bibitem{Lun10Com}
A. Lunardi,
\newblock {\em Compactness and asymptotic behavior in nonautonomous
linear parabolic equations with unbounded coefficients in $\R^d$}\,,
\newblock {arXiv: 1006.1530.}

\bibitem{MetPalWac02}
G. Metafune, D. Pallara, M. Wacker,
\newblock {\em Feller semigroups on $\R^N$}\,,
\newblock {Semigroup Forum}, \textbf{65} (2002), 159-205.


\bibitem{MetPalWac02 Com}
G. Metafune, D. Pallara, M. Wacker,
\newblock {\em Compactness Properties of Feller Semigroups}\,,
\newblock {Studia Math.}, \textbf{153} (2002), 179-206.

\bibitem{Stu94Har}
K.T. Sturm,
\newblock {\em Harnack's inequality for parabolic operators with singular low order
terms}\,,
\newblock {Math. Z.} 216 (1994), 593-612.
\end{thebibliography}
\end{document}